\documentclass[11pt,reqno]{amsart}

\usepackage{amsmath,amssymb,amsthm}
\usepackage{geometry}
\usepackage{microtype}
\usepackage[hidelinks]{hyperref}

\numberwithin{equation}{section}

\newtheorem{theorem}{Theorem}[section]
\newtheorem{proposition}[theorem]{Proposition}
\newtheorem{lemma}[theorem]{Lemma}
\newtheorem{corollary}[theorem]{Corollary}
\theoremstyle{remark}
\newtheorem{remark}[theorem]{Remark}

\title[Neumann problems for graph scalar curvature]
{Robin and Neumann problems for the graph scalar curvature equation}

\author{Guohuan Qiu}
\address{
State Key Laboratory of Mathematical Sciences and Institute of Mathematics,
Academy of Mathematics and Systems Science, Chinese Academy of Sciences,
55 Zhongguancun East Road, Beijing 100190, China}
\email{qiugh@amss.ac.cn}
\date{}

\hypersetup{
  pdftitle={Robin and Neumann problems for the graph scalar curvature equation},
  pdfauthor={Guohuan Qiu},
  pdfsubject={Existence and estimates for graph scalar curvature equations
    and an isoperimetric application},
  pdfkeywords={graph scalar curvature equation, Robin problem,
    Neumann problem, Alexandrov--Fenchel inequality,
    isoperimetric inequality, Newton transformation}
}


\begin{document}

\begin{abstract}
We study Robin and Neumann problems for the scalar curvature equation
of admissible graphs over bounded uniformly convex domains in three
dimensions. Under a small-volume assumption, we prove existence and
uniqueness for the Robin problem and obtain a classical Neumann
solution as the Robin parameter tends to zero. The volume threshold is
optimal among conditions depending only on the volume. The main step
is a boundary second-derivative estimate uniform in the Robin
parameter; known interior and global-to-boundary curvature estimates
then give the global bound.
\end{abstract}

\maketitle

\section{Introduction}

Let $\Omega$ be a bounded domain in $\mathbb R^3$ and let $u$ be a function on
$\overline\Omega$.  The graph $X(x)=(x,u(x))$ has upward unit normal
\[
                         N=\frac{(-Du,1)}{W},
             \qquad W=(1+|Du|^2)^{1/2}.
\]
We use the standard graph notation
\begin{equation}\label{eq:graph-tensors}
 g_{ij}=\delta_{ij}+u_i u_j,
 \qquad
 g^{ij}=\delta_{ij}-\frac{u_i u_j}{W^2},
 \qquad
 h_{ij}=\frac{u_{ij}}W,
 \qquad
 h_i{}^j=g^{jk}h_{ki}.
\end{equation}
For \(0<a\leq1\), we study the boundary value problem
\begin{equation}\label{eq:intro-problem}
 \begin{cases}
  \sigma _2(h_i{}^j)=1&\text{in }\Omega,\\
  u_\nu=-a u&\text{on }\partial\Omega,
 \end{cases}
\end{equation}
under the admissibility condition
\begin{equation}\label{eq:intro-admissible}
                 \kappa=(\kappa_1,\kappa_2,\kappa_3)\in\Gamma _2.
\end{equation}
Here $\nu$ denotes the outer unit normal to $\partial\Omega$ and
\[
 \Gamma _2=\{\kappa\in\mathbb R^3:\sigma _1(\kappa)>0,
                                  \ \sigma _2(\kappa)>0\}.
\]
This is the ellipticity cone for $\sigma_2$.

Up to the conventional factor \(2\), \(\sigma_2(h_i{}^j)\) is the
intrinsic scalar curvature of the graph. Thus the first equation in
\eqref{eq:intro-problem} prescribes constant graph scalar curvature.

We prove the two existence results below together with a global
second-derivative estimate that is uniform for \(0<a\leq1\).

\begin{theorem}[The Robin problem]\label{thm:existence}
Let \(0<\alpha<1\), and let
\(\Omega\subset\mathbb R^3\) be a bounded uniformly convex domain with
\(C^{5,\alpha}\) boundary.  If
\begin{equation}\label{eq:volume-condition}
                         |\Omega|<4\pi\sqrt3,
\end{equation}
then \eqref{eq:intro-problem} has a unique
\(\Gamma_2\)-admissible solution
\(u\in C^{4,\alpha}(\overline\Omega)\).
\end{theorem}

\begin{theorem}[The classical Neumann problem]
\label{thm:classical-neumann}
Under the hypotheses of Theorem~\ref{thm:existence}, there exists a
unique constant \(\lambda>0\) and, for this constant, a
\(\Gamma_2\)-admissible function
\(u\in C^{4,\alpha}(\overline\Omega)\), unique up to an additive
constant, such that
\begin{equation}\label{eq:classical-neumann}
 \begin{cases}
  \sigma _2(h_i{}^j)=1&\text{in }\Omega,\\
  u_\nu=\lambda&\text{on }\partial\Omega.
 \end{cases}
\end{equation}
\end{theorem}

\begin{theorem}[Boundary second-derivative estimate]
\label{thm:boundary-C2}
Let \(\Omega\subset\mathbb R^3\) be a bounded uniformly convex domain
with \(C^5\) boundary satisfying \eqref{eq:volume-condition}.  Every
admissible solution \(u\in C^4(\overline\Omega)\) of
\eqref{eq:intro-problem}, with \(0<a\leq1\), satisfies
\begin{equation}\label{eq:main-estimate}
 \max_{\overline\Omega}|D^2u|\leq C(\Omega).
\end{equation}
The constant is independent of \(a\in(0,1]\) and depends only on the
volume gap in \eqref{eq:volume-condition}, the fixed \(C^5\) geometry,
and the uniform convexity constants of \(\Omega\).
\end{theorem}

\begin{remark}[Optimality of the volume threshold]
\label{rem:optimal-volume}
The constant \(4\pi\sqrt3\) is optimal among hypotheses involving only
\(|\Omega|\) that guarantee solvability for every uniformly convex
domain.  Indeed, no classical admissible solution exists on \(B_R\)
when \(R\geq\sqrt3\).  For \(0<R<\sqrt3\) and \(0<a\leq1\), however,
\begin{equation}\label{eq:ball-cap}
 u_{R,a}(x)=-\sqrt{3-|x|^2}+\sqrt{3-R^2}
        -\frac{R}{a\sqrt{3-R^2}}
\end{equation}
is a strictly convex solution of the Robin problem on \(B_R\).
The same spherical cap, with an arbitrary vertical translation,
solves \eqref{eq:classical-neumann} with
\(\lambda=R/\sqrt{3-R^2}\).  Since
\(|B_{\sqrt3}|=4\pi\sqrt3\), the endpoint cannot be increased.  The
normal derivative diverges as \(R\uparrow\sqrt3\), in agreement with
the estimates proved below.
\end{remark}

\begin{remark}[Further questions]\label{rem:future-problems}
The method requires an interior curvature estimate. It may therefore
extend to the special Lagrangian curvature equation at the critical
phase or for convex solutions, where such an estimate is available
\cite{QiuZhouSLCE}. The examples of Qiu and Tao
\cite{QiuTaoSLCECounterexamples} indicate possible curvature blow-up
outside these settings. For general \(\sigma_k\)-curvature equations,
boundary second-derivative estimates for admissible Neumann solutions
remain open.
\end{remark}

\begin{corollary}[Endpoint Alexandrov--Fenchel inequality]
\label{cor:AF}
Let $0<\alpha<1$, and let
$\Omega\subset\mathbb R^3$ be a bounded uniformly convex domain
with $C^{5,\alpha}$ boundary.  If $\Pi$ is the second fundamental form
of $\partial\Omega$ with respect to the outer normal, then
\begin{equation}\label{eq:AF-intro}
 |\partial\Omega|^3
 \geq 9|\Omega|^2\int_{\partial\Omega}\sigma_2(\Pi)\,dS
 =36\pi|\Omega|^2.
\end{equation}
Equality holds if and only if $\Omega$ is a ball.
\end{corollary}

The Dirichlet theory for Hessian equations was developed by
Caffarelli, Nirenberg, and Spruck \cite{CNS}; see also Ivochkina
\cite{IvochkinaHessian} and Trudinger \cite{TrudingerHessian}.
For graph curvature equations, the prescribed Gauss curvature and
more general Weingarten curvature problems were studied by Trudinger
and Urbas \cite{TrudingerUrbasGauss}, Caffarelli, Nirenberg, and
Spruck \cite{CNSWeingarten}, Ivochkina \cite{IvochkinaCurvature},
and Trudinger \cite{TrudingerCurvature}, under their respective
structural and boundary convexity assumptions.

Neumann and oblique problems require different boundary arguments.
Lions, Trudinger, and Urbas \cite{LTU} established a priori estimates
and classical solvability for equations of Monge--Amp\`ere type.
Trudinger \cite{TrudingerBall} treated Hessian equations in balls and
conjectured the corresponding result for smooth uniformly convex
domains; Ma and Qiu \cite{MaQiu} resolved this
conjecture by proving global estimates through second order and the
existence of classical admissible solutions. For the real
special Lagrangian equation, Chen, Ma, and Wei \cite{ChenMaWei}
treated the supercritical phase and Wang \cite{WangCritical} the
critical phase; Qiu and Zhang \cite{QiuZhang} considered a related
class of special Lagrangian type equations. In particular, their work gives a
direct boundary double-normal estimate for the three-dimensional
\(2\)-Hessian equation. These are Hessian-eigenvalue equations rather
than graph principal-curvature equations.

For graph curvature equations, Urbas \cite{UrbasCurvature} proved
smooth solvability for two-dimensional nonuniformly elliptic equations
with nonlinear oblique data. Ma and Xu \cite{MaXu} obtained gradient
estimates and existence results for prescribed mean curvature
equations, Ma, Wang, and Wei \cite{MaWangWei} studied constant mean
curvature with nonzero Neumann data, and Deng and Ma
\cite{DengMaGradient} proved gradient estimates for higher-order
curvature equations with prescribed contact angle. The latter
condition does not include \(u_\nu=-au\). Neumann problems for convex
graphs have also been studied by curvature flows
\cite{SchnurerSchwetlick,SchnurerSmoczyk,XiaoCurvatureFlow}, but these
results do not give the Robin boundary \(C^2\) estimate needed here for
possibly nonconvex \(\Gamma_2\)-admissible graphs.

For the scalar curvature equation, Guan, Ren, and Wang
\cite{GuanRenWang} proved global curvature estimates for star-shaped
\(\Gamma_2\)-admissible hypersurfaces; in the graph setting, their
argument gives a global-to-boundary reduction. Guan and Qiu
\cite{GuanQiuInterior} obtained interior estimates in the convex case,
and the author \cite{QiuInterior} subsequently removed the convexity
assumption for three-dimensional admissible graphs. Related boundary
\(C^2\) estimates for conformal \(\sigma_2\)-curvature equations were
developed by Jin, Li, and Li \cite{JinLiLiBoundary}, S.-Y.~S. Chen
\cite{SChenConformalBoundary}, and Chen and Wei
\cite{ChenWeiSigmaTwoBoundary,ChenWeiModifiedSigmaTwo}. Their
tangential second-derivative and tangential-trace test functions partly
motivate the argument below.

The main difficulty is to obtain the boundary second-derivative estimate.
The Lions--Trudinger--Urbas reduction \cite{LTU} to the double-normal
derivative does not apply directly because the graph shape operator
depends on \(Du\), so the normal and tangential second derivatives
remain coupled. The global-to-boundary estimate of Guan--Ren--Wang
\cite{GuanRenWang} and Qiu's interior estimate \cite{QiuInterior}
reduce the remaining analysis to a fixed boundary collar.

In the collar, exponential barriers for the Robin residual control
\(u_{\nu\nu}\). This follows the strategy of Qiu--Zhang
\cite{QiuZhang}; the exponential term has the sign needed to absorb
the drift arising from the \(Du\)-dependence.
After homogenizing the boundary condition, tangential differentiation
and boundary semiconvexity reduce the remaining estimate to the
projected tangential trace. Differentiated curvature identities, a
quantitative concavity estimate for \(\sigma_2\), and a collar
coercivity inequality then control all second-order terms;
uniform boundary convexity supplies the final Hopf-type inequality.
A weighted ABP argument gives the height estimate, with
\eqref{eq:volume-condition} used only at this step, while a
boundary-adapted angle function gives the gradient estimate. The
resulting bounds yield the Robin solution by continuity and the
classical Neumann solution by letting \(a\downarrow0\) after mean
normalization.

Cabr\'e \cite{CabreABP} used the Poisson Neumann problem and the ABP
method to prove the Euclidean isoperimetric inequality. Trudinger
\cite{TrudingerQuermass} developed a nonlinear PDE approach to
quermassintegrals, and Qiu and Xia \cite{QiuXiaAF} used Neumann
solutions of Hessian equations and higher-order Reilly identities to
prove Alexandrov--Fenchel inequalities. Here the constant-Neumann
solution gives the three-dimensional endpoint by a graph-curvature
contact-set argument: ABP controls the Neumann constant and the graph
mean-curvature flux controls the boundary area.

Section~\ref{sec:preliminaries} contains the basic identities and the
height and gradient estimates. Sections~\ref{sec:normal}--
\ref{sec:completion} prove the boundary second-derivative estimate.
Section~\ref{sec:existence} treats optimality, existence, and
uniqueness, and Section~\ref{sec:AF} proves
Corollary~\ref{cor:AF} and related geometric inequalities.

\section{Preliminaries and first-order estimates}\label{sec:preliminaries}

\subsection{Geometric conventions and linearization}

We use the summation convention. Latin indices range from $1$ to $3$;
Greek indices range from $1$ to $2$ and denote tangential directions at
the boundary.  Indices on graph tensors are raised and lowered with
the graph metric \(g_{ij}\). Unless explicitly stated otherwise,
tensor norms and contractions are taken with respect to \(g_{ij}\).
Our sign convention for the second
fundamental form of \(\partial\Omega\subset\mathbb R^3\) is
\begin{equation}\label{eq:Pi-convention}
\begin{aligned}
 D_\tau\nu&=\Pi(\tau,e_\alpha)e_\alpha,
 & Dd&=-\nu,\\
 D^2d(\tau,\tau)&=-\Pi(\tau,\tau),
 & \Pi(\tau,\tau)&\geq\kappa_0|\tau|^2
 \quad(\tau\perp\nu).
\end{aligned}
\end{equation}
Here $d$ denotes the inward distance to the boundary.  In a fixed
tubular neighborhood we use its normal extension
\begin{equation}\label{eq:normal-extension}
                         n=-Dd.
\end{equation}
Thus \(|n|=1\), \(n=\nu\) on \(\partial\Omega\), and
\(D_n n=0\). Unless stated otherwise, \(C\) may change from line to
line and depends only on the volume gap in
\eqref{eq:volume-condition}, a fixed \(C^5\) boundary atlas, the
uniform convexity constant, and the associated tubular-neighborhood
bounds. A subscript records dependence on an auxiliary parameter.

For the graph \(X=(x,u(x))\), let
\begin{equation}\label{eq:newton-definition}
 [T_k]_i{}^j
 =
 \frac1{k!}
 \delta_{j j_1\cdots j_k}^{i i_1\cdots i_k}
 h_{j_1}{}^{i_1}\cdots h_{j_k}{}^{i_k},
 \qquad
 [T_k]^{ij}=[T_k]_p{}^i g^{pj}.
\end{equation}
Here and below \(\sigma_k=\sigma_k(\kappa_1,\kappa_2,\kappa_3)\).
Thus
\begin{equation}\label{eq:first-newton-tensor}
 [T_0]^{ij}=g^{ij},
 \qquad
 [T_1]_i{}^j=\sigma_1\delta_i{}^j-h_i{}^j,
 \qquad
 [T_1]^{ij}=\sigma_1g^{ij}-h^{ij}.
\end{equation}
The Newton transformations are symmetric and divergence free:
\begin{equation}\label{eq:newton-divergence}
             \nabla_i[T_k]_j{}^i=0,
             \qquad \nabla_i[T_k]^{ij}=0.
\end{equation}
They also satisfy the recurrence
\begin{equation}\label{eq:newton-recurrence}
 [T_k]_i{}^j
 =
 \sigma_k\delta_i{}^j-[T_{k-1}]_i{}^\ell h_\ell{}^j.
\end{equation}

When \(\sigma_2(h_i{}^j)\) is viewed as a function of
\((D^2u,Du)\), we denote its coordinate linearization on the base
domain by
\begin{equation}\label{eq:linearized-coefficients}
 F^{ij}=\frac{\partial\sigma_2(h_i{}^j)}{\partial u_{ij}},
 \qquad
 F_{u_i}=\frac{\partial\sigma_2(h_i{}^j)}{\partial u_i}.
\end{equation}

\begin{lemma}\label{lem:first-linearization}
Along an admissible solution of \(\sigma_2(h_i{}^j)=1\),
\begin{align}
 F^{ij}
 &=
 \frac1W[T_1]^{ij},                                      \label{eq:Fij}\\
 F_{u_i}
 &=
 2\bigl([T_2]^{ij}-2g^{ij}\bigr)u_j                      \label{eq:Fui}\\
 &=
 -2u^i-2[T_1]^{ij}h_{jk}u^k.                             \notag
\end{align}
In particular, \((F^{ij})\) is positive definite.  For every
\(\phi\in C^2(\Omega)\), its full coordinate linearization has the
geometric form
\begin{equation}\label{eq:newton-linearized-operator}
 F^{ij}\phi_{ij}+F_{u_i}\phi_i
 =
 \partial_i\left(\frac{[T_1]^{ij}}W\phi_j\right)
 =
 [T_1]^{ij}
 \left\{
 \nabla_i\nabla_j\left(\frac \phi W\right)
 +\frac \phi W h_i{}^k h_{kj}
 \right\}.
\end{equation}
Moreover, for every fixed \(m\),
\begin{equation}\label{eq:Fui-uim}
 F_{u_i}u_{im}
 =
 -4W u^i h_{im}+2W\sigma_3u_m.
\end{equation}
\end{lemma}

\begin{proof}
Hold \(Du\) fixed and vary the symmetric Hessian by a symmetric matrix
\(B=(B_{ij})\).  Then

\[
                    \delta h_{ij}=\frac1W B_{ij},
             \qquad \delta\sigma_2=\frac1W[T_1]^{ij}B_{ij}.
\]
Since this holds for every symmetric \(B\), it proves
\eqref{eq:Fij}.

For a variation of \(Du\) with \(D^2u\) fixed,
\[
 \delta h_{ij}=-h_{ij}u^k\delta u_k,
 \qquad
 \delta g_{ij}=u_i\delta u_j+u_j\delta u_i.
\]
Writing \(\sigma_2\) as a function of the mixed tensor \(h_i{}^j\)
therefore gives
\[
 \delta\sigma_2
 =
 [T_1]^{ij}\delta h_{ij}
 -[T_1]^{ik}h_k{}^j\delta g_{ij}.
\]
The identities
\[
 [T_1]^{ij}h_{ij}=2\sigma_2=2,
 \qquad
 [T_1]^{ik}h_k{}^j=g^{ij}-[T_2]^{ij}
\]
yield \eqref{eq:Fui}.

For the Jacobi form, consider the vertical variation
\(X_s(x)=(x,u(x)+s\phi(x))\).  Its normal speed is \(\phi/W\), while its
tangential part differentiates the constant function
\(\sigma_2(h_i{}^j)=1\) to zero.  The standard variation formula for
the Weingarten tensor, with the convention \(h_{ij}=u_{ij}/W\), gives
\[
 \left.\frac d{ds}\right|_{s=0}\sigma_2(h_i{}^j)
 =
 [T_1]^{ij}
 \left\{
 \nabla_i\nabla_j\left(\frac \phi W\right)
 +\frac \phi W h_i{}^k h_{kj}
 \right\}.
\]
The left-hand side is
\(F^{ij}\phi_{ij}+F_{u_i}\phi_i\).  Since
\(\det(g_{ij})=W^2\), the divergence-free identity for
\([T_1]^{ij}\), together with the Jacobi identity obtained by taking
\(\phi=1\), gives
\[
 [T_1]^{ij}
 \left\{
 \nabla_i\nabla_j\left(\frac \phi W\right)
 +\frac \phi W h_i{}^kh_{kj}
 \right\}
 =W\nabla_i\left(\frac{[T_1]^{ij}}{W^2}\phi_j\right)
 =\partial_i\left(\frac{[T_1]^{ij}}W\phi_j\right).
\]
This proves \eqref{eq:newton-linearized-operator}.

Finally, multiply \eqref{eq:Fui} by \(u_{im}=Wh_{im}\).  The Newton
recurrences give
\[
 [T_2]_i{}^j h_j{}^m=\sigma_3\delta_i{}^m.
\]
Consequently,
\[
 F_{u_i}u_{im}
 =
 2W\{[T_2]^{ij}-2g^{ij}\}u_jh_{im}
 =
 -4W u^i h_{im}+2W\sigma_3u_m,
\]
which is \eqref{eq:Fui-uim}.
\end{proof}

\begin{lemma}\label{lem:basic-algebra}
Let
\(\kappa_1\geq\kappa_2\geq\kappa_3\),
\(\kappa\in\Gamma_2\), and \(\sigma_2(\kappa)=1\).  Then
\begin{equation}\label{eq:ordered-curvatures}
 \kappa_1\geq\kappa_2>0,
 \qquad
 |\kappa_3|\leq\kappa_2,
 \qquad
 \kappa_i+\kappa_j>0\quad(i\ne j).
\end{equation}
For \(\{i,j,k\}=\{1,2,3\}\),
\begin{equation}\label{eq:component-polynomial}
 (\sigma_1-\kappa_i)(1+\kappa_i^2)
 =
 (\kappa_j+\kappa_k)(1+\kappa_i^2)
 =
 \sigma_1-\sigma_3.
\end{equation}
Consequently,
\begin{align}
 \sum_i(\sigma_1-\kappa_i)(1+\kappa_i^2)
 &=3(\sigma_1-\sigma_3),                                  \label{eq:trace-polynomial}\\
 \sigma_1-\sigma_3
 &=(\kappa_1+\kappa_2)(\kappa_1+\kappa_3)
                         (\kappa_2+\kappa_3)                \notag\\
 &=\sqrt{(1+\kappa_1^2)(1+\kappa_2^2)(1+\kappa_3^2)}.       \label{eq:three-pairwise-sums}
\end{align}
There are universal constants \(c,C>0\) such that
\begin{align}
 c(1+\kappa_1)
 &\leq \sigma_1-\sigma_3
 \leq C(1+\kappa_1)^3,                                     \label{eq:curvature-growth}\\
 \kappa_1^2\kappa_2^2+\kappa_1^2\kappa_3^2
                  +\kappa_2^2\kappa_3^2
 &\leq C(1+\kappa_1)(\sigma_1-\sigma_3),                    \label{eq:curvature-cofactor-growth}\\
 \sum_i(\sigma_1-\kappa_i)^2(1+\kappa_i^2)
 &=2\sigma_1(\sigma_1-\sigma_3)
 \leq C(1+\kappa_1)(\sigma_1-\sigma_3),                    \label{eq:newton-square-growth}\\
 \sigma_1-\sigma_3
 &\leq C(1+\kappa_1)
       \left(1+\min_i\kappa_i^2\right).                     \label{eq:least-curvature-growth}
\end{align}
\end{lemma}

\begin{proof}
The standard description of \(\Gamma_2\) in dimension three gives
\eqref{eq:ordered-curvatures}.  Since \(\sigma_2=1\),
\begin{align*}
 (\kappa_j+\kappa_k)(1+\kappa_i^2)
 &=(\kappa_j+\kappa_k)
   \{\kappa_i(\kappa_i+\kappa_j+\kappa_k)+\kappa_j\kappa_k\}\\
 &=\sigma_1-\sigma_3,
\end{align*}
which proves \eqref{eq:component-polynomial}.  Summation gives
\eqref{eq:trace-polynomial}.  The first identity in
\eqref{eq:three-pairwise-sums} follows from
\[
 (\kappa_1+\kappa_2)(\kappa_1+\kappa_3)(\kappa_2+\kappa_3)
 =\sigma_1\sigma_2-\sigma_3,
\]
and the second follows by taking the modulus of
\[
 \prod_{i=1}^3(1+\sqrt{-1}\,\kappa_i)
 =(1-\sigma_2)+\sqrt{-1}\,(\sigma_1-\sigma_3).
\]

If \(\kappa_3<0\), then
\[
 -\kappa_3
 <\frac{\kappa_1\kappa_2}{\kappa_1+\kappa_2}
 \leq\frac{\kappa_1}{2};
\]
if \(\kappa_3\geq0\), the estimates below are immediate.  Furthermore,
\[
 1
 =\kappa_1(\kappa_2+\kappa_3)+\kappa_2\kappa_3
 \leq
 \kappa_1(\kappa_2+\kappa_3)
 +\frac14(\kappa_2+\kappa_3)^2,
\]
and hence
\[
 \kappa_2+\kappa_3
 \geq2\{\sqrt{\kappa_1^2+1}-\kappa_1\}
 \geq\frac{c}{1+\kappa_1}.
\]
These facts and \eqref{eq:three-pairwise-sums} prove
\eqref{eq:curvature-growth}.

Next,
\[
 \kappa_1^2\kappa_2^2
 =\{1-\kappa_3(\kappa_1+\kappa_2)\}^2
 \leq(1+\kappa_3^2)\{1+(\kappa_1+\kappa_2)^2\}
 \leq C(1+\kappa_1)(\sigma_1-\sigma_3),
\]
where the last step follows by dividing by the square-root identity in
\eqref{eq:three-pairwise-sums} and using
\(|\kappa_3|\leq\kappa_2\leq\kappa_1\).  The same ordering then proves
\eqref{eq:curvature-cofactor-growth}.
Multiplying \eqref{eq:component-polynomial} by
\(\sigma_1-\kappa_i\) and summing gives the exact identity in
\eqref{eq:newton-square-growth}, because
\(\sum_i(\sigma_1-\kappa_i)=2\sigma_1\).  Finally, choose an
index at which \(|\kappa_i|\) is least in
\eqref{eq:component-polynomial}; this gives
\eqref{eq:least-curvature-growth}.
\end{proof}

\subsection{The height estimate}\label{subsec:height}

The scale condition enters only through the following weighted ABP
estimate for \(a u\), which is uniform as \(a\downarrow0\).

\begin{lemma}[Height estimate]\label{lem:height-estimate}
Let \(u\in C^2(\overline\Omega)\) be an admissible solution of
\eqref{eq:intro-problem}.  Then \(u\leq0\) in
\(\overline\Omega\).  If \eqref{eq:volume-condition} holds, then
\begin{equation}\label{eq:C0-estimate}
 a\|u\|_{C^0(\overline\Omega)}
 \leq
 \bigl(1+a\operatorname{diam}\Omega\bigr)
 \frac{\left(\dfrac{|\Omega|}{4\pi\sqrt3}\right)^{1/3}}
 {\sqrt{1-\left(\dfrac{|\Omega|}{4\pi\sqrt3}\right)^{2/3}}}.
\end{equation}
\end{lemma}

\begin{proof}
An interior maximum of \(u\) is impossible: there \(Du=0\) and
\(D^2u\leq0\), so \(\sigma_1(\kappa)\leq0\), contrary to
\(\kappa\in\Gamma_2\).  Thus the maximum is attained on
\(\partial\Omega\).  At a boundary maximum \(u_\nu\geq0\), and
the boundary condition gives \(u\leq0\).

Since \(\sigma_2=1\), the solution is not identically zero.  Set
\(m=-\min_{\overline\Omega}u>0\), and let \(x_0\) be a minimum point.
It lies in \(\Omega\), since at a boundary minimum one would
have \(u_\nu(x_0)\leq0\), whereas \(u_\nu(x_0)=a m>0\).  Put
\[
 r_0=\frac{a m}{1+a\operatorname{diam}\Omega}.
\]
For \(|p|<r_0\), let \(y\) minimize
\(x\mapsto u(x)-p\cdot(x-x_0)\) on \(\overline\Omega\).  If
\(y\in\partial\Omega\), then
\[
 0\geq\bigl(u-p\cdot(x-x_0)\bigr)_\nu(y)
       =-a u(y)-p\cdot\nu(y),
\]
so \(-u(y)\leq|p|/a\).  Comparison with \(x_0\) would then give
\[
 m\leq-u(y)+p\cdot(y-x_0)
 \leq\bigl(a^{-1}+\operatorname{diam}\Omega\bigr)|p|<m,
\]
a contradiction.  Hence \(y\in\Omega\).

Let
\[
 \mathcal C=\{x\in\Omega:
 u(z)\geq u(x)+Du(x)\cdot(z-x)
 \text{ for every }z\in\overline\Omega\}
\]
be the lower contact set.  The preceding argument gives
\(B_{r_0}(0)\subset Du(\mathcal C)\).  On \(\mathcal C\) one has
\(D^2u\geq0\).  Thus \(h_{ij}=u_{ij}/W\) is positive semidefinite
with respect to the graph metric, and the principal curvatures are
nonnegative there.  Since
\(\sigma_2(\kappa)=1\), the Newton--Maclaurin inequality yields
\[
 \sigma_3(\kappa)\leq
 \left(\frac{\sigma_2(\kappa)}3\right)^{3/2}
 =\frac1{3\sqrt3}.
\]
Moreover, \(\det(g_{ij})=W^2\) and \(h_{ij}=u_{ij}/W\), so
\begin{equation}\label{eq:contact-determinant}
 \sigma_3(\kappa)=\det(h_i{}^j)
 =\frac{\det D^2u}{W^5}
 \qquad\text{on }\mathcal C.
\end{equation}
The area formula, with multiplicity, gives
\begin{align*}
 \int_{B_{r_0}(0)}\frac{dp}{(1+|p|^2)^{5/2}}
 &\leq\int_{\mathcal C}
 \frac{\det D^2u}{(1+|Du|^2)^{5/2}}\,dx\\
 &=\int_{\mathcal C}\sigma_3(\kappa)\,dx
 \leq\frac{|\Omega|}{3\sqrt3}.
\end{align*}
Since
\[
 \int_{B_{r_0}(0)}\frac{dp}{(1+|p|^2)^{5/2}}
 =\frac{4\pi}{3}\frac{r_0^3}{(1+r_0^2)^{3/2}},
\]
we obtain
\[
 \frac{r_0}{\sqrt{1+r_0^2}}
 \leq\left(\frac{|\Omega|}{4\pi\sqrt3}\right)^{1/3}.
\]
The volume condition allows this inequality to be inverted.  Using
\(a m=(1+a\operatorname{diam}\Omega)r_0\) proves
\eqref{eq:C0-estimate}.
\end{proof}

\subsection{The global gradient estimate}\label{subsec:gradient}

We use a boundary correction of \(\log W\) that cancels
\(u_{\nu\nu}\) on \(\partial\Omega\).

\begin{proposition}\label{prop:global-gradient}
Let \(\Omega\subset\mathbb R^3\) be a bounded uniformly convex domain
with \(C^3\) boundary satisfying \eqref{eq:volume-condition}. If
\(u\in C^3(\overline\Omega)\) is an admissible solution of
\eqref{eq:intro-problem}, then
\begin{equation}\label{eq:global-gradient-estimate}
 \|Du\|_{L^\infty(\Omega)}\leq C(\Omega).
\end{equation}
The constant is independent of \(a\in(0,1]\) and of
\(\|u\|_{C^0(\overline\Omega)}\).
\end{proposition}

\begin{proof}
Extend the vector field \(n\) from \eqref{eq:normal-extension}
smoothly to a neighborhood of \(\overline\Omega\), without changing it
in a fixed tubular neighborhood of the boundary, and set
\[
                         \bar n(x,x_4)=(n(x),0).
\]
Let
\[
                         X(x)=(x,u(x))
\]
be the position vector of the graph, let \(N\) be its upward unit
normal, and let \(E_4=(0,0,0,1)\). We denote by \(\overline D\) the
Euclidean connection of \(\mathbb R^4\), and use a local
graph-orthonormal frame \(E_1,E_2,E_3\).

By Lemma~\ref{lem:height-estimate},
\begin{equation}\label{eq:uniform-au-bound}
                         \|a u\|_{C^0(\overline\Omega)}
                         \leq C(\Omega).
\end{equation}
Consider
\begin{equation}\label{eq:gradient-test-function}
 \Psi
 =
 \log W+\frac{a u(Du\cdot n+a u)}{W^2}.
\end{equation}

We first compute the full linearization of \(\Psi\). By
\eqref{eq:newton-linearized-operator}, for every scalar function
\(\phi\) on the graph,
\begin{align}
 F^{ij}\phi_{ij}+F_{u_i}\phi_i
 ={}&
 \frac1W[T_1]^{ij}\nabla_i\nabla_j\phi
 -2[T_1]^{ij}h_i{}^k(\nabla_k u)(\nabla_j\phi).
 \label{eq:gradient-full-linearization}
\end{align}
Consequently, at an interior critical point of \(\Psi\),
\begin{equation}\label{eq:gradient-critical-linearization}
 F^{ij}\Psi_{ij}+F_{u_i}\Psi_i
 =
 \frac1W[T_1]^{ij}\nabla_i\nabla_j\Psi.
\end{equation}

The geometric quantities entering \(\Psi\) can be written as
\[
 u=\langle X,E_4\rangle,
 \qquad
 \frac1W=\langle N,E_4\rangle,
 \qquad
 -\frac{Du\cdot n}{W}=\langle N,\bar n\rangle.
\]
The Gauss and Weingarten formulas give
\begin{align}
 \nabla_i\nabla_j u&=\frac{h_{ij}}W,
 &
 \nabla_i\left(\frac1W\right)
 &=-h_i{}^k\nabla_k u.                                    \label{eq:vertical-angle-identities}
\end{align}
Using
\[
 [T_1]^{ij}h_{ij}=2,
 \qquad
 [T_1]^{ij}h_i{}^kh_{kj}=\sigma_1-3\sigma_3,
 \qquad
 [T_1]^{ij}\nabla_i h_{jk}=0,
\]
we obtain
\[
 [T_1]^{ij}\nabla_i\nabla_j\left(\frac1W\right)
 =
 -\frac{\sigma_1-3\sigma_3}{W}.
\]
It follows that
\begin{align}
 [T_1]^{ij}\nabla_i\nabla_j\log W
 ={}&
 \sigma_1-3\sigma_3 \notag\\
 &+
 W^2[T_1]^{ij}h_i{}^k(\nabla_k u)
                    h_j{}^\ell(\nabla_\ell u).
 \label{eq:log-W-geometric-Hessian}
\end{align}

We next compute the two factors in the correction term. First,
\begin{align}
 [T_1]^{ij}\nabla_i\nabla_j\left(\frac{a u}{W}\right)
 ={}&
 -\frac{a u}{W}(\sigma_1-3\sigma_3)
 +\frac{2a}{W^2} \notag\\
 &-2a[T_1]^{ij}h_i{}^k(\nabla_k u)(\nabla_j u).
 \label{eq:gradient-first-factor-hessian}
\end{align}
For the angle with the extended normal field,
\begin{equation}\label{eq:normal-angle-first-derivative}
 \nabla_i\left(-\frac{Du\cdot n}{W}\right)
 =
 -h_i{}^k\langle E_k,\bar n\rangle
 +\langle N,\overline D_{E_i}\bar n\rangle.
\end{equation}
Differentiating once more and using
\([T_1]^{ij}\nabla_i h_{jk}=0\), we obtain
\begin{align}
 &[T_1]^{ij}\nabla_i\nabla_j
       \left(-\frac{Du\cdot n}{W}\right)                 \notag\\
 &\quad=
 \frac{Du\cdot n}{W}(\sigma_1-3\sigma_3)
 -[T_1]^{ij}h_j{}^k
       \langle E_k,\overline D_{E_i}\bar n\rangle
 -[T_1]^{ij}h_i{}^k
       \langle E_k,\overline D_{E_j}\bar n\rangle        \notag\\
 &\qquad
 +[T_1]^{ij}
       \bigl\langle N,(\overline D^2\bar n)(E_i,E_j)\bigr\rangle
 +2\langle N,\overline D_N\bar n\rangle.
 \label{eq:normal-angle-second-derivative}
\end{align}
Lemma~\ref{lem:basic-algebra}, the fixed bounds for the extension
\(\bar n\), and
\[
                         [T_1]^{ij}h_i{}^kh_{kj}
                         =\sigma_1-3\sigma_3
\]
give
\begin{equation}\label{eq:normal-angle-error}
 \left|
 [T_1]^{ij}\nabla_i\nabla_j
       \left(-\frac{Du\cdot n}{W}\right)
 -\frac{Du\cdot n}{W}(\sigma_1-3\sigma_3)
 \right|
 \leq C_\Omega(\sigma_1-3\sigma_3).
\end{equation}
The same estimates imply
\begin{align}
 &[T_1]^{ij}
 \langle N,\overline D_{E_i}\bar n\rangle
 \langle N,\overline D_{E_j}\bar n\rangle
 +[T_1]^{ij}(\nabla_i u)(\nabla_j u)                       \notag\\
 &\quad
 +[T_1]^{ij}h_i{}^k\langle E_k,\bar n\rangle
                  h_j{}^\ell\langle E_\ell,\bar n\rangle
 \leq C_\Omega(\sigma_1-3\sigma_3).
 \label{eq:fixed-field-bounds}
\end{align}

Combining \eqref{eq:gradient-first-factor-hessian} and
\eqref{eq:normal-angle-error}, we have
\begin{align}
 &\left|
 [T_1]^{ij}\nabla_i\nabla_j
       \left(\frac{a u+Du\cdot n}{W}\right)
 +\frac{a u+Du\cdot n}{W}(\sigma_1-3\sigma_3)
 -\frac{2a}{W^2}\right.                                  \notag\\
 &\hspace{28mm}\left.
 +2a[T_1]^{ij}h_i{}^k(\nabla_k u)(\nabla_j u)
 \right|
 \leq C_\Omega(\sigma_1-3\sigma_3).
 \label{eq:gradient-second-factor-hessian}
\end{align}

The product rule gives
\begin{align}
 &[T_1]^{ij}\nabla_i\nabla_j
       \left(\frac{a u(Du\cdot n+a u)}{W^2}\right)          \notag\\
 &\quad=
 \underbrace{
 \frac{a u+Du\cdot n}{W}
 [T_1]^{ij}\nabla_i\nabla_j\left(\frac{a u}{W}\right)
 }_{I_1}                                                   \notag\\
 &\qquad+
 \underbrace{
 \frac{a u}{W}
 [T_1]^{ij}\nabla_i\nabla_j
       \left(\frac{a u+Du\cdot n}{W}\right)
 }_{I_2}                                                   \notag\\
 &\qquad+
 \underbrace{
 2[T_1]^{ij}
 \nabla_i\left(\frac{a u}{W}\right)
 \nabla_j\left(\frac{a u+Du\cdot n}{W}\right)
 }_{I_3}.
 \label{eq:gradient-correction-exact}
\end{align}
By \eqref{eq:uniform-au-bound} and \(|Du\cdot n|\leq W\),
\begin{equation}\label{eq:gradient-correction-small}
 \left|\frac{a u}{W}\right|\leq\frac{C_\Omega}{W},
 \qquad
 \left|\frac{a u+Du\cdot n}{W}\right|\leq C_\Omega.
\end{equation}
Using \eqref{eq:gradient-first-factor-hessian}, we write
\begin{align*}
 I_1
 ={}&
 -\frac{a u(a u+Du\cdot n)}{W^2}
       (\sigma_1-3\sigma_3)
 +\frac{2a(a u+Du\cdot n)}{W^3}\\
 &-\frac{2a(a u+Du\cdot n)}{W}
       [T_1]^{ij}h_i{}^k(\nabla_k u)(\nabla_j u).
\end{align*}
The weighted Cauchy--Schwarz inequality and
\eqref{eq:fixed-field-bounds} give, for every \(\varepsilon>0\),
\begin{align*}
 &\frac{2a|a u+Du\cdot n|}{W}
 \left|
 [T_1]^{ij}h_i{}^k(\nabla_k u)(\nabla_j u)
 \right|\\
 &\qquad\leq
 \varepsilon W^2[T_1]^{ij}h_i{}^k(\nabla_k u)
                         h_j{}^\ell(\nabla_\ell u)
 +\frac{C_{\varepsilon,\Omega}}{W^2}
       (\sigma_1-3\sigma_3).
\end{align*}
Consequently,
\begin{align}
 I_1
 \geq{}&
 -\varepsilon W^2[T_1]^{ij}h_i{}^k(\nabla_k u)
                         h_j{}^\ell(\nabla_\ell u)
 -\frac{C_{\varepsilon,\Omega}}{W}
       (\sigma_1-3\sigma_3).
 \label{eq:gradient-I1-estimate}
\end{align}

Similarly, \eqref{eq:gradient-second-factor-hessian} gives
\begin{align*}
 I_2
 \geq{}&
 -\frac{a u(a u+Du\cdot n)}{W^2}
       (\sigma_1-3\sigma_3)
 +\frac{2a^2u}{W^3}\\
 &-\frac{2a^2u}{W}
       [T_1]^{ij}h_i{}^k(\nabla_k u)(\nabla_j u)
 -\frac{C_\Omega|a u|}{W}(\sigma_1-3\sigma_3).
\end{align*}
Another application of weighted Cauchy--Schwarz therefore yields
\begin{align}
 I_2
 \geq{}&
 -\varepsilon W^2[T_1]^{ij}h_i{}^k(\nabla_k u)
                         h_j{}^\ell(\nabla_\ell u)
 -\frac{C_{\varepsilon,\Omega}}{W}
       (\sigma_1-3\sigma_3).
 \label{eq:gradient-I2-estimate}
\end{align}

It remains to estimate \(I_3\). Direct differentiation gives
\begin{align}
 \nabla_i\left(\frac{a u}{W}\right)
 &=
 \frac{a\nabla_i u}{W}
 -a u h_i{}^k\nabla_k u,                                \label{eq:gradient-correction-first-one}\\
 \nabla_i\left(\frac{a u+Du\cdot n}{W}\right)
 &=
 \nabla_i\left(\frac{a u}{W}\right)
 +h_i{}^k\langle E_k,\bar n\rangle
 -\langle N,\overline D_{E_i}\bar n\rangle.
 \label{eq:gradient-correction-first-two}
\end{align}
Hence
\begin{align*}
 I_3
 ={}&
 2[T_1]^{ij}
 \nabla_i\left(\frac{a u}{W}\right)
 \nabla_j\left(\frac{a u}{W}\right)\\
 &+
 2[T_1]^{ij}
 \left\{
 \frac{a\nabla_i u}{W}
 -a u h_i{}^k\nabla_k u
 \right\}\\
 &\hspace{20mm}\times
 \left\{
 h_j{}^\ell\langle E_\ell,\bar n\rangle
 -\langle N,\overline D_{E_j}\bar n\rangle
 \right\}.
\end{align*}
The first term is nonnegative because \([T_1]\) is positive definite.
By \eqref{eq:fixed-field-bounds},
\begin{align*}
 &\left|
 \frac{2a}{W}[T_1]^{ij}(\nabla_i u)
 \left\{
 h_j{}^\ell\langle E_\ell,\bar n\rangle
 -\langle N,\overline D_{E_j}\bar n\rangle
 \right\}
 \right|\\
 &\qquad\leq
 \frac{C_\Omega}{W}(\sigma_1-3\sigma_3).
\end{align*}
Moreover, Young's inequality gives
\begin{align*}
 &2|a u|
 \left|
 [T_1]^{ij}h_i{}^k(\nabla_k u)
 \left\{
 h_j{}^\ell\langle E_\ell,\bar n\rangle
 -\langle N,\overline D_{E_j}\bar n\rangle
 \right\}
 \right|\\
 &\qquad\leq
 \varepsilon W^2[T_1]^{ij}h_i{}^k(\nabla_k u)
                         h_j{}^\ell(\nabla_\ell u)\\
 &\qquad\quad+
 \frac{C_{\varepsilon}|a u|^2}{W^2}
 [T_1]^{ij}
 \left\{
 h_i{}^k\langle E_k,\bar n\rangle
 -\langle N,\overline D_{E_i}\bar n\rangle
 \right\}\\
 &\hspace{57mm}\times
 \left\{
 h_j{}^\ell\langle E_\ell,\bar n\rangle
 -\langle N,\overline D_{E_j}\bar n\rangle
 \right\}\\
 &\qquad\leq
 \varepsilon W^2[T_1]^{ij}h_i{}^k(\nabla_k u)
                         h_j{}^\ell(\nabla_\ell u)
 +\frac{C_{\varepsilon,\Omega}}{W^2}
       (\sigma_1-3\sigma_3).
\end{align*}
We conclude that
\begin{align}
 I_3
 \geq{}&
 -\varepsilon W^2[T_1]^{ij}h_i{}^k(\nabla_k u)
                         h_j{}^\ell(\nabla_\ell u)
 -\frac{C_{\varepsilon,\Omega}}{W}
       (\sigma_1-3\sigma_3).
 \label{eq:gradient-quadratic-absorption}
\end{align}

Combining
\eqref{eq:gradient-I1-estimate},
\eqref{eq:gradient-I2-estimate}, and
\eqref{eq:gradient-quadratic-absorption}, and replacing
\(\varepsilon\) by \(\varepsilon/3\), we obtain
\begin{align}
 &[T_1]^{ij}\nabla_i\nabla_j
       \left(\frac{a u(Du\cdot n+a u)}{W^2}\right)          \notag\\
 &\qquad\geq
 -\varepsilon W^2[T_1]^{ij}h_i{}^k(\nabla_k u)
                         h_j{}^\ell(\nabla_\ell u)
 -\frac{C_{\varepsilon,\Omega}}{W}
       (\sigma_1-3\sigma_3).
 \label{eq:gradient-correction-lower}
\end{align}

At an interior maximum of \(\Psi\),
\eqref{eq:gradient-critical-linearization},
\eqref{eq:log-W-geometric-Hessian}, and
\eqref{eq:gradient-correction-lower} give
\begin{align}
 F^{ij}\Psi_{ij}+F_{u_i}\Psi_i
 \geq\frac1W\Biggl\{&
 \left(1-\frac{C_{\varepsilon,\Omega}}W\right)
       (\sigma_1-3\sigma_3) \notag\\
 &+(1-\varepsilon)
 W^2[T_1]^{ij}h_i{}^k(\nabla_k u)
                    h_j{}^\ell(\nabla_\ell u)
 \Biggr\}.
 \label{eq:gradient-coercivity}
\end{align}
Fix the \(\varepsilon\) in
\eqref{eq:gradient-correction-lower} to be \(1/4\); equivalently, use
\(1/12\) in each of the three preceding Young inequalities.  The
constant \(C_{\varepsilon,\Omega}\) is then determined solely by the
fixed data and \eqref{eq:uniform-au-bound}.  Next choose
\[
                         W_0\geq
                         \max\{2C_{\varepsilon,\Omega},2\}.
\]
By Lemma~\ref{lem:basic-algebra}, the right-hand side of
\eqref{eq:gradient-coercivity} is strictly positive whenever
\(W\geq W_0\). This contradicts the maximum principle,
since at an interior maximum
\[
                         F^{ij}\Psi_{ij}+F_{u_i}\Psi_i\leq0.
\]
Thus an interior maximum of \(\Psi\) can occur only where
\(W\leq W_0(\Omega)\).

It remains to consider a boundary maximum. On \(\partial\Omega\),
\[
                         Du\cdot n+a u=0,
\]
so the correction term in \eqref{eq:gradient-test-function} vanishes.
Tangential differentiation of \(u_\nu=-a u\) gives
\begin{equation}\label{eq:mixed-gradient-boundary-identity}
 u_{\nu\alpha}
 =
 -a u_\alpha-\Pi_{\alpha\beta}u_\beta.
\end{equation}
Since \(n=\nu\) on \(\partial\Omega\) and \(D_\nu n=0\),
\begin{align*}
 \Psi_\nu
 &=
 \frac{W_\nu}{W}
 +\frac{a u}{W^2}
   \bigl(Du\cdot n+a u\bigr)_\nu\\
 &=
 \frac{\sum_i u_i u_{i\nu}}{W^2}
 +\frac{a u}{W^2}
 \left(
 u_{\nu\nu}+Du\cdot D_\nu n+a u_\nu
 \right)\\
 &=
 \frac{\sum_i u_i u_{i\nu}}{W^2}
 +\frac{a u}{W^2}\left(u_{\nu\nu}-a^2u\right).
\end{align*}
With respect to an orthonormal boundary frame
\(e_1,e_2,\nu\),
\begin{align*}
 \sum_i u_i u_{i\nu}
 &=
 \sum_{\alpha=1}^2u_\alpha u_{\alpha\nu}
 +u_\nu u_{\nu\nu}\\
 &=
 -a\sum_{\alpha=1}^2u_\alpha^2
 -\Pi_{\alpha\beta}u_\alpha u_\beta
 -a u\,u_{\nu\nu}.
\end{align*}
Therefore,
\begin{align}
 \Psi_\nu
 &=
 \frac1{W^2}
 \left\{
 -a\sum_{\alpha=1}^2u_\alpha^2
 -\Pi_{\alpha\beta}u_\alpha u_\beta
 -a u\,u_{\nu\nu}
 +a u\,u_{\nu\nu}
 -a^3u^2
 \right\} \notag\\
 &=
 -\frac{\displaystyle
       a\sum_{\alpha=1}^2u_\alpha^2
       +\Pi_{\alpha\beta}u_\alpha u_\beta
       +a^3u^2}
       {W^2}.
 \label{eq:gradient-boundary-normal}
\end{align}
Thus the terms containing \(u_{\nu\nu}\) cancel exactly. Uniform
convexity makes the right-hand side strictly negative whenever
\(W>1\). At a boundary maximum, however, the outward normal derivative
is nonnegative. Hence a boundary maximum with \(W>1\) is impossible.

Finally, \eqref{eq:uniform-au-bound} gives
\[
 \left|
 \frac{a u(Du\cdot n+a u)}{W^2}
 \right|
 \leq C(\Omega)
 \quad\text{on }\overline\Omega.
\]
At a maximum point of \(\Psi\), either \(W\leq W_0(\Omega)\), or the
point lies on the boundary and \(W=1\). Hence
\[
                         \max_{\overline\Omega}\Psi\leq C(\Omega).
\]
Comparing this with a point where \(W\) attains its maximum gives
\[
                         \sup_\Omega W\leq C(\Omega),
\]
and \eqref{eq:global-gradient-estimate} follows.
\end{proof}

Under the hypotheses of Theorem~\ref{thm:boundary-C2}, the preceding
estimates give
\(\|a u\|_{C^0}+\|Du\|_{L^\infty}\leq C(\Omega)\), so the graph and
Euclidean metrics are uniformly equivalent.  Hence
\begin{equation}\label{eq:hessian-curvature-comparison}
 1+|D^2u|\asymp1+\kappa_1,
\end{equation}
and \eqref{eq:Fij}, \eqref{eq:trace-polynomial} give
\begin{equation}\label{eq:linearized-growth}
 1+\sum_iF^{ii}+\sum_{i,j,k}F^{ij}u_{ik}u_{jk}
 \asymp1+\sigma_1-\sigma_3.
\end{equation}
Here \(A\asymp B\) means comparability by positive constants depending
only on the data in Theorem~\ref{thm:boundary-C2}.

\section{The double-normal estimate}\label{sec:normal}

\subsection{A contact estimate}

\begin{lemma}\label{lem:normal-contact}
At a point of an admissible solution, suppose that
\[
 C_0^{-1}\leq|\xi|\leq C_0,
 \qquad |u_{ij}\xi_j-\chi\xi_i|\leq C_0.
\]
Then
\begin{align}
 &\sum_i(\sigma_1-\kappa_i)|\kappa_i|+|\sigma_3|
 +\left|F_{u_i}u_{ij}\xi_j\right|                          \notag\\
 &\qquad\leq
 C\left\{\sum_iF^{ii}
          +\chi^2F^{ij}\xi_i\xi_j\right\}.
 \label{eq:contact-bound}
\end{align}
The constant depends only on \(C_0\) and an upper bound for \(|Du|\).
\end{lemma}

\begin{proof}
In Euclidean coordinates the contact relation reads
\(Wh_{ij}\xi_j=\chi\xi_i+O(1)\).  Contracting its square with the
positive tensor \([T_1]^{ij}\), and using the uniform equivalence of
the graph and Euclidean metrics, gives
\begin{align}
 [T_1]^{pq}h_{pi}\xi_i h_{qj}\xi_j
 \leq
 C\left\{\chi^2F^{ij}\xi_i\xi_j
          +\sum_iF^{ii}\right\}.
 \label{eq:contact-component-square}
\end{align}

Choose a principal frame orthonormal with respect to \(g_{ij}\), and
write \(\xi=\xi^i e_i\).  Some component \(|\xi^j|\) is bounded below
by a positive constant depending only on \(C_0\) and the assumed upper
bound for \(|Du|\).  The
left-hand side of \eqref{eq:contact-component-square} then contains
\((\sigma_1-\kappa_j)\kappa_j^2|\xi^j|^2\).  Since
\[
 \sigma_1-\sigma_3
 =(\sigma_1-\kappa_j)(1+\kappa_j^2)
\]
and \(\sum_iF^{ii}\) is uniformly comparable to
\(\sum_i(\sigma_1-\kappa_i)\), we obtain
\begin{equation}\label{eq:contact-curvature-bound}
 \sigma_1-\sigma_3
 \leq
 C\left\{\sum_iF^{ii}
          +\chi^2F^{ij}\xi_i\xi_j\right\}.
\end{equation}
Moreover,
\[
 \sum_i(\sigma_1-\kappa_i)|\kappa_i|
 \leq
 \frac12\sum_i(\sigma_1-\kappa_i)(1+\kappa_i^2)
 =\frac32(\sigma_1-\sigma_3),
\]
and \(|\sigma_3|\leq\sigma_1-\sigma_3\).

Taking the Euclidean scalar product of the assumed relation with
\(\xi\) yields
\[
 \chi|\xi|^2=W h_{ij}\xi_i\xi_j+O(1).
\]
Since \(|h|\leq C\sigma_1\) in \(\Gamma_2\) and
\(2\sigma_1=[T_1]^{ij}g_{ij}\), we obtain
\begin{equation}\label{eq:contact-chi}
 |\chi|\leq C\left(1+\sum_iF^{ii}\right)
 \leq C\sum_iF^{ii}.
\end{equation}
Here the last inequality follows from
\(\sigma_1^2\geq3\sigma_2=3\) and the uniform equivalence of the graph
and Euclidean metrics.

Finally, contract \eqref{eq:Fui-uim} with \(\xi_m\).  Because
\(Wu^ih_{im}=W^{-2}u_i u_{im}\),
\[
 F_{u_i}u_{ij}\xi_j
 =
 -\frac4{W^2}u_i\{\chi\xi_i+O(1)\}
 +2W\sigma_3u_j\xi_j.
\]
The gradient estimate in Proposition~\ref{prop:global-gradient},
\eqref{eq:contact-curvature-bound},
and \eqref{eq:contact-chi} prove \eqref{eq:contact-bound}.
\end{proof}

\subsection{Exponential barriers}

\begin{proposition}\label{prop:double-normal}
Assume the hypotheses of Theorem~\ref{thm:boundary-C2}. Then
\[
                         \max_{\partial\Omega}|u_{\nu\nu}|\leq C.
\]
\end{proposition}

\begin{proof}
We first fix
\[
                              \varepsilon=\frac{\kappa_0}{4}.
\]
We then choose \(\delta>0\), depending only on the fixed tubular
geometry, so small that \(d\) is smooth in \(\{0<d<\delta\}\),
\[
 2\varepsilon\delta\leq\frac12,
 \qquad
 -\sum_{i,j}d_{ij}\xi_i\xi_j
 \geq\frac{\kappa_0}{2}|\xi|^2
 \quad\text{whenever }\xi\cdot Dd=0.
\]
In this collar set
\[
                              \rho=-d+\varepsilon d^2.
\]
Since \(D^2d(Dd,\cdot)=0\) and
\[
 \rho_{ij}=(-1+2\varepsilon d)d_{ij}
        +2\varepsilon d_i d_j,
\]
the preceding choices give
\begin{equation}\label{eq:convex-defining-function}
 \begin{gathered}
 \rho=0,\qquad \rho_i=\nu_i\quad\text{on }\partial\Omega,\\
 \rho<0,\qquad \frac12\leq |D\rho|\leq2,\\
 \sum_{i,j}\rho_{ij}\xi_i\xi_j\geq\theta|\xi|^2
 \quad\text{in }\{0<d<\delta\}.
 \end{gathered}
\end{equation}
Here one may take \(\theta=\kappa_0/4\).

Set
\[
                              G=\sum_j u_j\rho_j+a u.
\]
By the boundary condition \(u_\nu=-a u\) and
\eqref{eq:convex-defining-function}, we have \(G=0\) on
\(\partial\Omega\). Moreover, Proposition~\ref{prop:global-gradient} gives
\[
                              |G|\leq R_0
              \quad\text{in }\{0<d<\delta\},
\]
where \(R_0\) depends only on the fixed data. For constants
\(\tau>0\) and \(B_0>0\), define
\[
 P_-=\frac{1-e^{-\tau G}}{\tau}-B_0\rho,
 \qquad
 P_+=\frac{e^{\tau G}-1}{\tau}+B_0\rho.
\]
Both functions vanish on \(\partial\Omega\).  Moreover,
\(-\rho\geq\delta/2\) on \(\{d=\delta\}\), while \(|G|\leq R_0\).
Hence, for fixed \(\tau>0\),
\[
                         P_->0,\qquad P_+<0
                         \quad\text{on }\{d=\delta\}
\]
whenever
\[
                         B_0>
 \frac{2(e^{\tau R_0}-1)}{\tau\delta}.
\]
Suppose first that \(P_-\) has a negative minimum at an interior point.
The first derivative condition gives
\[
                         e^{-\tau G}G_i=B_0\rho_i,
\]
and hence
\begin{equation}\label{eq:lower-contact}
 \sum_j u_{ij}\rho_j
 =
 B_0e^{\tau G}\rho_i
 -
 \left(\sum_j u_j\rho_{ji}+a u_i\right).
\end{equation}
The expression in parentheses is uniformly bounded by the fixed data
in view of Proposition~\ref{prop:global-gradient}.

Differentiating the equation \(F(D^2u,Du)=1\) once gives
\begin{equation}\label{eq:first-differentiated-equation}
 \sum_{i,j}F^{ij}u_{ij\ell}
 +\sum_iF_{u_i}u_{i\ell}
 =0.
\end{equation}
A direct differentiation of \(G\), followed by the use of
\eqref{eq:first-differentiated-equation}, yields
\begin{align}
 &\sum_{i,j}F^{ij}G_{ij}+\sum_iF_{u_i}G_i \notag\\
 &\quad=
 2\sum_{i,j,\ell}F^{ij}u_{\ell i}\rho_{\ell j}
 +\sum_{i,j,\ell}F^{ij}u_\ell \rho_{\ell ij}
 +a\sum_{i,j}F^{ij}u_{ij} \notag\\
 &\qquad
 +\sum_iF_{u_i}
       \left(\sum_j u_j\rho_{ji}+a u_i\right).
 \label{eq:linearized-residual}
\end{align}
At the minimum point, \eqref{eq:lower-contact} implies
\begin{align}
 &e^{-\tau G}\sum_iF_{u_i}
       \left(\sum_j u_j\rho_{ji}+a u_i\right)
 -B_0\sum_iF_{u_i}\rho_i \notag\\
 &\qquad
 =-e^{-\tau G}\sum_{i,j}F_{u_i}u_{ij}\rho_j.
 \label{eq:lower-drift-cancellation}
\end{align}
Consequently, \eqref{eq:linearized-residual} and
\eqref{eq:lower-drift-cancellation} give
\begin{align}
 &\sum_{i,j}F^{ij}(P_-)_{ij}
 +\sum_iF_{u_i}(P_-)_i \notag\\
 &\quad=
 e^{-\tau G}
 \Biggl\{
 2\sum_{i,j,\ell}F^{ij}u_{\ell i}\rho_{\ell j}
 +\sum_{i,j,\ell}F^{ij}u_\ell \rho_{\ell ij}
 +a\sum_{i,j}F^{ij}u_{ij}-\sum_{i,j}F_{u_i}u_{ij}\rho_j
 \Biggr\} \notag\\
 &\qquad
 -\tau e^{-\tau G}\sum_{i,j}F^{ij}G_i G_j
 -B_0\sum_{i,j}F^{ij}\rho_{ij}.
 \label{eq:residual-linearization}
\end{align}

At the contact point, \eqref{eq:lower-contact} has the form required in
Lemma~\ref{lem:normal-contact}, with coefficient
\(B_0e^{\tau G}\) and a uniformly bounded error term.  By
\eqref{eq:Fij} and \(u_{ij}=Wh_{ij}\), we have the exact identity
\begin{equation}\label{eq:normal-matrix-change}
 \sum_{i,j,\ell}F^{ij}u_{\ell i}\rho_{\ell j}
 =\sum_{i,j,\ell}[T_1]^{ij}h_{\ell i}\rho_{\ell j}.
\end{equation}
In a principal frame orthonormal with respect to \(g_{ij}\), the
contraction in \eqref{eq:normal-matrix-change} is the pairing of
\([T_1]h\), whose eigenvalues are
\((\sigma_1-\kappa_i)\kappa_i\), with the uniformly bounded tensor
obtained from \(D^2\rho\) by raising its first index with the Euclidean
metric.  This metric need not be the identity in this frame.
Consequently,
\begin{equation}\label{eq:normal-matrix-bound}
 \left|\sum_{i,j,\ell}F^{ij}u_{\ell i}\rho_{\ell j}\right|
 \leq C\sum_i(\sigma_1-\kappa_i)|\kappa_i|.
\end{equation}
Thus \eqref{eq:contact-bound} gives
\begin{align}
 &\left|
 \sum_{i,j,\ell}F^{ij}u_{\ell i}\rho_{\ell j}
 \right|
 +
 \left|
 \sum_{i,j}F_{u_i}u_{ij}\rho_j
 \right| \notag\\
 &\qquad\leq
 C\sum_iF^{ii}
 +C_0B_0^2e^{2\tau G}
      \sum_{i,j}F^{ij}\rho_i \rho_j,
 \label{eq:lower-contact-estimate}
\end{align}
where \(C_0\) depends only on the fixed geometry and the already
established gradient bound.  In particular, it is independent of
\(a\), of the solution, and of \(\tau\) and \(B_0\).

The remaining terms in braces in
\eqref{eq:residual-linearization} are bounded by
\(C\sum_iF^{ii}\). Indeed, the derivatives of \(\rho\) are fixed, the
first derivatives of \(u\) are bounded, and the homogeneity of \(F\)
in its Hessian argument gives
\begin{equation}\label{eq:euler-full-operator}
 \sum_{i,j}F^{ij}u_{ij}
 =
 2F(D^2u,Du)
 =
 2.
\end{equation}
Moreover, admissibility and Proposition~\ref{prop:global-gradient} imply
\[
                              \sum_iF^{ii}\geq c>0,
\]
so all bounded zeroth-order terms may be absorbed into
\(C\sum_iF^{ii}\).

At the minimum point,
\[
                         G_i=B_0e^{\tau G}\rho_i.
\]
Combining \eqref{eq:convex-defining-function},
\eqref{eq:residual-linearization}, and
\eqref{eq:lower-contact-estimate}, we obtain
\begin{align}
 0
 &\leq
 \sum_{i,j}F^{ij}(P_-)_{ij}
 +\sum_iF_{u_i}(P_-)_i \notag\\
 &\leq
 C_\tau\sum_iF^{ii}
 +(C_0-\tau)B_0^2e^{\tau G}
       \sum_{i,j}F^{ij}\rho_i \rho_j
 -\theta B_0\sum_iF^{ii}.
 \label{eq:lower-normal-barrier}
\end{align}
Thus a negative interior minimum of \(P_-\) is impossible if
\(\tau>C_0\) and \(\theta B_0>C_\tau\).

Suppose next that \(P_+\) has a positive interior maximum. The first
derivative condition gives
\[
                         e^{\tau G}G_i=-B_0\rho_i,
\]
and hence
\begin{equation}\label{eq:upper-contact}
 \sum_j u_{ij}\rho_j
 =
 -B_0e^{-\tau G}\rho_i
 -
 \left(\sum_j u_j\rho_{ji}+a u_i\right).
\end{equation}
The drift terms cancel exactly as in
\eqref{eq:lower-drift-cancellation}.  Applying
Lemma~\ref{lem:normal-contact} to \eqref{eq:upper-contact} gives
\begin{align}
 0
 &\geq
 \sum_{i,j}F^{ij}(P_+)_{ij}
 +\sum_iF_{u_i}(P_+)_i \notag\\
 &\geq
 -C_\tau\sum_iF^{ii}
 +(\tau-C_0)B_0^2e^{-\tau G}
       \sum_{i,j}F^{ij}\rho_i \rho_j
 +\theta B_0\sum_iF^{ii}.
 \label{eq:upper-normal-barrier}
\end{align}
Increase \(C_0\), if necessary, so that the same constant works in
\eqref{eq:lower-normal-barrier} and
\eqref{eq:upper-normal-barrier}, and choose
\[
                              \tau>C_0+1.
\]
Since \(|G|\leq R_0\), the exponential factors and all constants
denoted by \(C_\tau\) are now fixed. Enlarge \(C_\tau\), if necessary,
and choose
\[
 B_0>
 \max\left\{
 \frac{2(e^{\tau R_0}-1)}{\tau\delta},
 \frac{C_\tau+1}{\theta}
 \right\}.
\]
The first condition gives the required signs on \(\{d=\delta\}\);
the second excludes both interior contacts. Consequently,
\[
                              P_-\geq0,
 \qquad
                              P_+\leq0
              \quad\text{in }\{0<d<\delta\}.
\]
The parameters have been chosen in the order
\[
             \varepsilon,\quad \delta\ \text{and }\theta,
             \quad \tau,\quad B_0.
\]
They depend only on the data and the preceding height and gradient
estimates.

On \(\partial\Omega\), both \(P_-\) and \(P_+\) vanish. Since
\(P_-\geq0\) and \(P_+\leq0\) on the interior side, their outward
normal derivatives satisfy
\[
                         (P_-)_\nu\leq0,
                         \qquad
                         (P_+)_\nu\geq0.
\]
Using \(G=0\) and \(\rho_\nu=1\) on \(\partial\Omega\), we obtain
\begin{equation}\label{eq:normal-residual-derivative}
                         -B_0\leq G_\nu\leq B_0.
\end{equation}
Finally, on \(\partial\Omega\),
\begin{align*}
 G_\nu
 &=
 \sum_{i,j}u_{ij}\nu_i\nu_j
 +\sum_{i,j}u_i \rho_{ij}\nu_j
 +a u_\nu\\
 &=
 u_{\nu\nu}
 +\sum_{i,j}u_i \rho_{ij}\nu_j-a^2u.
\end{align*}
Every term in the last line except \(u_{\nu\nu}\) is bounded by
\eqref{eq:uniform-au-bound}, the fixed boundary geometry, and
Proposition~\ref{prop:global-gradient}. The desired
estimate now follows from \eqref{eq:normal-residual-derivative}.
\end{proof}

The barriers also give a collar estimate for the boundary residual.

\begin{corollary}\label{cor:residual-collar}
For the functions used in Proposition~\ref{prop:double-normal},
\begin{equation}\label{eq:residual-distance}
                 |u_i \rho_i+a u|\leq Cd
\end{equation}
in a fixed boundary collar.
\end{corollary}

\begin{proof}
Since \(|G|\leq R_0\), the derivatives of both exponential profiles
are bounded above and below. The barrier inequalities and
\(-\rho=d+O(d^2)\) give \eqref{eq:residual-distance}, with \(C\)
independent of \(a\).
\end{proof}

\section{Tangential boundary estimates}\label{sec:boundary-trace}

Retain the distance function \(d\) and normal field \(n\) fixed above,
and set
\begin{equation}\label{eq:homogenization}
                         v=(1-a d)u.
\end{equation}
Since \(d_\nu=-1\),
\begin{equation}\label{eq:v-normal-zero}
                         v_\nu=0\qquad\text{on }\partial\Omega.
\end{equation}
Together with tangential differentiation of the boundary condition,
the double-normal estimate reduces the boundary Hessian estimate to
the projected tangential trace of \(v\).

\subsection{Boundary reduction}

\begin{lemma}\label{lem:boundary-semiconvexity}
For every Euclidean orthonormal tangential frame on
\(\partial\Omega\),
\begin{equation}\label{eq:tangential-semiconvexity}
                  (u_{\alpha\beta})\geq-CI,
             \qquad (v_{\alpha\beta})\geq-CI.
\end{equation}
Consequently,
\begin{equation}\label{eq:boundary-reconstruction}
 |D^2u|
 \leq C\left(
 1+\bigl[(\delta_{ij}-\nu_i\nu_j)v_{ij}\bigr]_+\right)
 \qquad\text{on }\partial\Omega.
\end{equation}
\end{lemma}

\begin{proof}
Fix a boundary point and choose an orthonormal frame with \(e_3=\nu\)
that diagonalizes the tangential block of \(D^2u\).
Equation~\eqref{eq:mixed-gradient-boundary-identity} bounds
\(u_{\alpha\nu}\), while Proposition~\ref{prop:double-normal} bounds
\(u_{33}\).

At the chosen point, expanding \(\sigma_2(h_i{}^j)=1\) using
\eqref{eq:graph-tensors} yields the following exact identity because
\(u_{12}=0\):
\begin{equation}\label{eq:boundary-equation-expansion}
\begin{aligned}
 &(1+u_3^2)u_{11}u_{22}\\
 &\quad+\bigl[(1+u_2^2)u_{33}-2u_2u_3u_{23}\bigr]u_{11}\\
 &\quad+\bigl[(1+u_1^2)u_{33}-2u_1u_3u_{13}\bigr]u_{22}\\
 &\quad-(1+u_1^2)u_{23}^2-(1+u_2^2)u_{13}^2
       +2u_1u_2u_{13}u_{23}=W^4 .
\end{aligned}
\end{equation}
The mixed and normal entries \(u_{\alpha3}\) and \(u_{33}\) are
bounded.  Positivity of the two linearized coefficients gives
\[
 \begin{aligned}
 &(1+u_3^2)u_{22}
 +(1+u_2^2)u_{33}-2u_2u_3u_{23}>0,\\
 &(1+u_3^2)u_{11}
 +(1+u_1^2)u_{33}-2u_1u_3u_{13}>0.
 \end{aligned}
\]
Since \(1+u_3^2\geq1\), both tangential eigenvalues have a uniform
lower bound.

Direct differentiation of \eqref{eq:homogenization} gives
\begin{align}
 v_{ij}={}&(1-a d)u_{ij}
 -a d_{ij} u
 -a d_i u_j
 -a d_j u_i.                         \label{eq:v-second-derivatives}
\end{align}
On the boundary, every component of \(D^2v-D^2u\) is bounded, which
proves the second inequality in \eqref{eq:tangential-semiconvexity}.
The same formula and Proposition~\ref{prop:double-normal} bound
\(v_{\nu\nu}\).  Differentiating \eqref{eq:v-normal-zero}
tangentially gives
\begin{equation}\label{eq:v-mixed}
                         v_{\alpha\nu}
                         =-\Pi_{\alpha\beta}v_\beta,
\end{equation}
so the mixed derivatives of \(v\) are bounded.  The two tangential
eigenvalues of \(D^2v\) have a common lower bound; their sum is
\((\delta_{ij}-\nu_i\nu_j)v_{ij}\).  This proves
\eqref{eq:boundary-reconstruction}.
\end{proof}

\begin{lemma}\label{lem:boundary-normal-trace}
There are fixed constants \(\vartheta>0\) and \(C\) such that
\begin{equation}\label{eq:Hopf-trace}
 \left\{(\delta_{ij}-n_in_j)v_{ij}\right\}_\nu
 \leq-\vartheta(\delta_{ij}-n_in_j)v_{ij}+C
 \qquad\text{on }\partial\Omega.
\end{equation}
\end{lemma}

\begin{proof}
At a boundary point choose a geodesic principal frame satisfying
\(D_{e_\alpha}\nu=\Pi_{\alpha\alpha}e_\alpha\).  Differentiating
\(v_\nu=0\) twice in the \(e_\alpha\)-direction gives
\[
 0=v_{\alpha\alpha\nu}-\Pi_{\alpha\alpha}v_{\nu\nu}
       +2\Pi_{\alpha\alpha}v_{\alpha\alpha}+O(1).
\]
Here \(O(1)\) consists of derivatives of the boundary second
fundamental form contracted with \(Dv\).
The normal extension in \eqref{eq:normal-extension} satisfies
\(D_\nu(\delta_{ij}-n_in_j)=0\) on the boundary.  Summing the preceding
identities, we obtain
\begin{equation}\label{eq:normal-trace-formula}
 \left\{\sum_{i,j} (\delta_{ij}-n_in_j)v_{ij}\right\}_\nu
 =-2\sum_{\alpha,\beta}\Pi^{\alpha\beta}v_{\alpha\beta}
  +(\operatorname{tr}\Pi)v_{\nu\nu}+O(1).
\end{equation}
The double normal derivative is bounded.  Uniform convexity and
Lemma~\ref{lem:boundary-semiconvexity} therefore give
\[
 -2\sum_{\alpha,\beta}\Pi^{\alpha\beta}v_{\alpha\beta}
 \leq-2\kappa_0(v_{11}+v_{22})+C
 =-2\kappa_0\sum_{i,j} (\delta_{ij}-n_in_j)v_{ij}+C,
\]
which proves \eqref{eq:Hopf-trace}.
\end{proof}

\subsection{Differentiated identities and the
\texorpdfstring{\(\sigma_2\)}{sigma2}-concavity estimate}

In this subsection all derivatives are covariant on the graph. Thus
\(h_{ijk}=\nabla_kh_{ij}\) and
\(h_{ijkl}=\nabla_l\nabla_kh_{ij}\); by Codazzi, \(h_{ijk}\) is fully
symmetric.  At a chosen point an orthonormal frame always refers to
the graph metric \(g_{ij}\).

\begin{lemma}\label{lem:geometric-component-second-variation}
Let \(\sigma_2(h_i{}^j)=1\).  In a frame orthonormal with respect to
\(g_{ij}\), and for each fixed \(m\),
\begin{align}
 [T_1]^{ij}h_{ijm}&=0,                                    \label{eq:first-geometric-variation}\\
 [T_1]^{ij}h_{ijmm}
 +\left(\sum_i h_{iim}\right)^2
 -\sum_{i,j}h_{ijm}^2&=0.                                 \label{eq:second-geometric-variation}
\end{align}
More generally, at the same point,
\begin{align}
 [T_1]^{ij}h_{k\ell ij}
 ={}&
 \sum_{r,s}h_{rsk}h_{rs\ell}
 -\left(\sum_rh_{rrk}\right)\left(\sum_sh_{ss\ell}\right)\notag\\
 &+2h_k{}^p h_{p\ell}-(\sigma_1-3\sigma_3)h_{k\ell}.
 \label{eq:newton-simons-component}
\end{align}
\end{lemma}

\begin{proof}
Choose a local geodesic frame, orthonormal with respect to \(g_{ij}\),
at the point under consideration.  The first differential of
\(\sigma_2\) is \([T_1]^{ij}h_{ijm}\).  In this frame its second
differential in the \(e_m\)-direction is
\[
 [T_1]^{ij}h_{ijmm}
 +(g^{ij}g^{k\ell}-g^{ik}g^{j\ell})h_{ijm}h_{k\ell m},
\]
which is precisely \eqref{eq:second-geometric-variation}.

Polarizing the second variation gives
\[
 [T_1]^{ij}h_{ijk\ell}
 =
 \sum_{r,s}h_{rsk}h_{rs\ell}
 -\left(\sum_rh_{rrk}\right)\left(\sum_sh_{ss\ell}\right).
\]
On the other hand, Codazzi together with the Gauss and Ricci
identities gives
\[
 [T_1]^{ij}(h_{k\ell ij}-h_{ijk\ell})
 =2h_k{}^ph_{p\ell}-(\sigma_1-3\sigma_3)h_{k\ell}.
\]
Adding these formulas proves
\eqref{eq:newton-simons-component}.
\end{proof}

Because the tangential projection need not commute with \(h_i{}^j\),
we diagonalize the projection and work in the resulting
graph-orthonormal frame.

\begin{lemma}[Quantitative \(\sigma_2\)-concavity]
\label{lem:sigma2-concavity}
Let \((h_{ij})\) be a symmetric \(3\times3\) matrix with ordered
eigenvalues
\[
 \kappa_1\geq\kappa_2\geq\kappa_3,
 \qquad
 \kappa\in\Gamma_2,
 \qquad
 \sigma_2(h_i{}^j)=1.
\]
Let \(e_1,e_2,e_3\) be an orthonormal frame, and let \(h_{ijk}\) be a
fully symmetric three-tensor satisfying
\begin{equation}\label{eq:sigma2-concavity-tangency}
 \sum_{i,j}[T_1]^{ij}h_{ij\alpha}=0,
 \qquad \alpha=1,2.
\end{equation}
Suppose that
\[
 0<c_0\leq\mu_\alpha\leq C_0,
 \qquad \alpha=1,2.
\]
Then, for arbitrary symmetric coefficient arrays
\((B_\alpha^{ij})\), \(\alpha=1,2\),
\begin{align}
 &\sum_{\alpha=1}^2\mu_\alpha
 \left\{
 \sum_{i,j}h_{ij\alpha}^2
 -\left(\sum_i h_{ii\alpha}\right)^2
 \right\}
 +\sum_{\alpha=1}^2\sum_{i,j}B_\alpha^{ij}h_{ij\alpha}
 \notag\\
 &\qquad\geq
 -C\sum_{\alpha=1}^2
 \left\{
 \sum_{i,j}(B_\alpha^{ij})^2
 +\frac{\left|\sum_{i,j}B_\alpha^{ij}h_{ij}\right|^2}
        {1+\min_i\kappa_i^2}
 \right\},
 \label{eq:sigma2-concavity-estimate}
\end{align}
where \(C\) depends only on \(c_0\) and \(C_0\).
\end{lemma}

\begin{proof}
At the point under consideration,
\[
 [T_1]^{ij}=\sigma_1\delta_{ij}-h_{ij}.
\]
Since \(\sigma_2=1\),
\begin{equation}\label{eq:newton-normal-identities}
 \sum_i[T_1]^{ii}=2\sigma_1,
 \qquad
 \sum_{i,j}\bigl([T_1]^{ij}\bigr)^2
 =2(\sigma_1^2-1).
\end{equation}
The tangency condition
\eqref{eq:sigma2-concavity-tangency} therefore gives
\begin{align*}
 \sum_i h_{ii\alpha}
 &=
 \sum_{i,j}
 \left(
 \delta_{ij}
 -\frac{\sigma_1}{\sigma_1^2-1}[T_1]^{ij}
 \right)h_{ij\alpha}.
\end{align*}
Thus, by Cauchy--Schwarz,
\[
\left(\sum_i h_{ii\alpha}\right)^2
\leq \sum_{i,j}
 \left(
 \delta_{ij}
 -\frac{\sigma_1}{\sigma_1^2-1}[T_1]^{ij}
 \right)^2
 \sum_{i,j}h_{ij\alpha}^2.
\]
A direct calculation using
\eqref{eq:newton-normal-identities} gives
\[
 \sum_{i,j}
 \left(
 \delta_{ij}
 -\frac{\sigma_1}{\sigma_1^2-1}[T_1]^{ij}
 \right)^2
 =
 \frac{\sigma_1^2-3}{\sigma_1^2-1}.
\]
Consequently,
\begin{equation}\label{eq:sigma2-tangent-lower-bound}
 \sum_{i,j}h_{ij\alpha}^2
 -\left(\sum_i h_{ii\alpha}\right)^2
 \geq
 \frac{2}{\sigma_1^2-1}
 \sum_{i,j}h_{ij\alpha}^2.
\end{equation}
If \(\sigma_1\) is bounded, this gives a uniform coercive estimate, and
\eqref{eq:sigma2-concavity-estimate} follows immediately from
Young's inequality.  We may therefore assume that \(\sigma_1\) is
larger than a fixed constant.

Set
\begin{equation}\label{eq:sigma2-reference-tensor}
 Z_{ij}=h_{ij}-\frac1{\sigma_1}\delta_{ij}.
\end{equation}
Then
\begin{equation}\label{eq:Z-properties}
 \sum_{i,j}[T_1]^{ij}Z_{ij}=0,
 \qquad
 \sum_{i,j}(Z_{ij}-h_{ij})^2=\frac3{\sigma_1^2},
\end{equation}
and
\begin{equation}\label{eq:Z-quadratic-value}
 \sum_{i,j}Z_{ij}^2-\left(\sum_iZ_{ii}\right)^2
 =
 2\left(1-\frac3{\sigma_1^2}\right)\geq c.
\end{equation}
For \(\alpha=1,2\), decompose
\begin{equation}\label{eq:sigma2-concavity-decomposition}
 h_{ij\alpha}=t_\alpha Z_{ij}+Y^\alpha_{ij},
\end{equation}
where
\begin{equation}\label{eq:sigma2-concavity-orthogonality}
 \sum_{i,j}[T_1]^{ij}Y^\alpha_{ij}=0,
 \qquad
 \sum_{i,j}Z_{ij}Y^\alpha_{ij}=0.
\end{equation}
Since
\[
 [T_1]^{ij}+Z_{ij}
 =
 \left(\sigma_1-\frac1{\sigma_1}\right)\delta_{ij},
\]
the two relations in \eqref{eq:sigma2-concavity-orthogonality} imply
\[
                         \sum_iY^\alpha_{ii}=0.
\]
It follows that
\begin{align}
 &\sum_{i,j}h_{ij\alpha}^2
 -\left(\sum_i h_{ii\alpha}\right)^2 \notag\\
 &\qquad=
 t_\alpha^2
 \left\{
 \sum_{i,j}Z_{ij}^2-\left(\sum_iZ_{ii}\right)^2
 \right\}
 +\sum_{i,j}(Y^\alpha_{ij})^2.
 \label{eq:sigma2-concavity-quadratic}
\end{align}

Full symmetry of \(h_{ijk}\) and
\eqref{eq:sigma2-concavity-decomposition} give
\begin{equation}\label{eq:sigma2-concavity-coupling}
 t_1Z_{i2}-t_2Z_{i1}
 =
 Y^2_{i1}-Y^1_{i2}.
\end{equation}
Since
\[
                         h_{ij}=Z_{ij}+\frac1{\sigma_1}\delta_{ij},
\]
equation \eqref{eq:sigma2-concavity-coupling} gives
\begin{align*}
 t_1h_{i2}-t_2h_{i1}
 =
 Y^2_{i1}-Y^1_{i2}
 +\frac1{\sigma_1}
   \bigl(t_1\delta_{i2}-t_2\delta_{i1}\bigr).
\end{align*}
The least singular value of \((h_{ij})\) is
\(\min_i|\kappa_i|\). Therefore,
\begin{align}
 \min_i\kappa_i^2\,(t_1^2+t_2^2)
 \leq{}&
 C\sum_{i,j}
 \left\{
 (Y^1_{ij})^2+(Y^2_{ij})^2
 \right\} \notag\\
 &+\frac{C}{\sigma_1^2}(t_1^2+t_2^2).
 \label{eq:sigma2-concavity-control}
\end{align}
After enlarging the fixed lower threshold for \(\sigma_1\),
\eqref{eq:Z-quadratic-value} gives a fixed positive lower bound for
the coefficient of \(t_\alpha^2\) in
\eqref{eq:sigma2-concavity-quadratic}.  To use
\eqref{eq:sigma2-concavity-control}, consider two cases.  If
\[
                 \min_i\kappa_i^2\leq\frac{2C}{\sigma_1^2},
\]
then \(\min_i\kappa_i^2(t_1^2+t_2^2)\) is controlled by the
\(t_\alpha^2\)-terms already present in
\eqref{eq:sigma2-concavity-quadratic}.  Otherwise, the last term on
the right-hand side of \eqref{eq:sigma2-concavity-control} is absorbed
into its left-hand side, and
\(\min_i\kappa_i^2(t_1^2+t_2^2)\) is controlled by the terms containing
\(Y^\alpha_{ij}\).  Thus \eqref{eq:Z-quadratic-value},
\eqref{eq:sigma2-concavity-quadratic}, and
\eqref{eq:sigma2-concavity-control}, together with the fixed bounds for
\(\mu_1,\mu_2\), give
\begin{align}
 &\sum_{\alpha=1}^2\mu_\alpha
 \left\{
 \sum_{i,j}h_{ij\alpha}^2
 -\left(\sum_i h_{ii\alpha}\right)^2
 \right\} \notag\\
 &\qquad\geq
 c\sum_{\alpha=1}^2\sum_{i,j}(Y^\alpha_{ij})^2
 +c\left(1+\min_i\kappa_i^2\right)(t_1^2+t_2^2).
 \label{eq:sigma2-concavity-coercivity}
\end{align}

Finally, \eqref{eq:sigma2-reference-tensor} and
\eqref{eq:sigma2-concavity-decomposition} give
\begin{align*}
 \sum_{i,j}B_\alpha^{ij}h_{ij\alpha}
 ={}&
 t_\alpha\sum_{i,j}B_\alpha^{ij}h_{ij}
 -\frac{t_\alpha}{\sigma_1}\sum_iB_\alpha^{ii}
 +\sum_{i,j}B_\alpha^{ij}Y^\alpha_{ij}.
\end{align*}
For every \(\varepsilon>0\), Young's inequality gives
\begin{align*}
 &\left|
 \sum_{i,j}B_\alpha^{ij}Y^\alpha_{ij}
 -\frac{t_\alpha}{\sigma_1}\sum_iB_\alpha^{ii}
 \right|\\
 &\qquad\leq
 \varepsilon\left\{
 \sum_{i,j}(Y^\alpha_{ij})^2+t_\alpha^2
 \right\}
 +C_\varepsilon\sum_{i,j}(B_\alpha^{ij})^2,
\end{align*}
and
\begin{align*}
 \left|
 t_\alpha\sum_{i,j}B_\alpha^{ij}h_{ij}
 \right|
 \leq{}&
 \varepsilon
 \left(1+\min_i\kappa_i^2\right)t_\alpha^2\\
 &+
 C_\varepsilon
 \frac{\left|\sum_{i,j}B_\alpha^{ij}h_{ij}\right|^2}
      {1+\min_i\kappa_i^2}.
\end{align*}
The lower threshold for \(\sigma_1\) and the coercivity constant in
\eqref{eq:sigma2-concavity-coercivity} are now fixed in terms of
\(c_0,C_0\).  Choose \(\varepsilon\) smaller than a fixed fraction of
that coercivity constant.  After summing over \(\alpha=1,2\), the terms
containing \(Y^\alpha_{ij}\) and \(t_\alpha\) are absorbed by
\eqref{eq:sigma2-concavity-coercivity}.  This proves
\eqref{eq:sigma2-concavity-estimate}.
\end{proof}

\subsection{Tangential trace estimates}

\begin{proposition}\label{prop:tangential-trace}
There is a fixed boundary collar, independent of \(a\in(0,1]\), in
which \(d\leq1/2\),
\[
                         \frac12\leq1-a d\leq\frac32,
\]
and
\begin{align}
 &F^{ij}
 \bigl[(1-a d)(\delta_{k\ell}-n_kn_\ell)u_{k\ell}\bigr]_{ij}
 +F_{u_i}
 \bigl[(1-a d)(\delta_{k\ell}-n_kn_\ell)u_{k\ell}\bigr]_i
 \notag\\
 &\quad=
 [T_1]^{ij}
 \Bigl\{
 \nabla_i\nabla_j
 \bigl[(1-a d)(\delta_{k\ell}-n_kn_\ell)h_{k\ell}\bigr]
 \notag\\
 &\hspace{43mm}
 +(1-a d)(\delta_{k\ell}-n_kn_\ell)h_{k\ell}
       h_i{}^p h_{pj}
 \Bigr\}                                                   \notag\\
 &\quad\geq
 -C(1+\kappa_1)(\sigma_1-\sigma_3).
 \label{eq:homogenization-absorption}
\end{align}
\end{proposition}

\begin{proof}
At a fixed point, regard
\[
                         P^{k\ell}=\delta_{k\ell}-n_kn_\ell
\]
as a symmetric contravariant tensor on the graph. Relative to the graph
metric, \(P\) is positive semidefinite of rank two, and its two nonzero
eigenvalues have fixed positive upper and lower bounds by
Proposition~\ref{prop:global-gradient}. Choose a graph-orthonormal frame
at the point such that
\[
                         P^{k\ell}
 =\operatorname{diag}(\mu_1,\mu_2,0),
 \qquad
 0<c\leq\mu_\alpha\leq C.
\]

Applying \eqref{eq:newton-linearized-operator} with
\[
 \phi=(1-a d)P^{k\ell}u_{k\ell}
\]
gives the equality in \eqref{eq:homogenization-absorption}, since
\[
                         \frac{\phi}{W}
 =(1-a d)P^{k\ell}h_{k\ell}.
\]
Expanding the right-hand side and using
\eqref{eq:newton-simons-component}, together with
\[
 [T_1]^{ij}h_i{}^p h_{pj}=\sigma_1-3\sigma_3,
\]
we obtain
\begin{align}
 &F^{ij}
 \bigl[(1-a d)P^{k\ell}u_{k\ell}\bigr]_{ij}
 +F_{u_i}
 \bigl[(1-a d)P^{k\ell}u_{k\ell}\bigr]_i                  \notag\\
 ={}&
 (1-a d)\sum_{\alpha=1}^2\mu_\alpha
 \left\{
 \sum_{r,s}h_{rs\alpha}^2
 -\left(\sum_rh_{rr\alpha}\right)^2
 +2\sum_p h_{\alpha p}^2
 \right\}                                                  \notag\\
 &+2(1-a d)[T_1]^{ij}(\nabla_iP)^{k\ell}h_{k\ell j}
 -2a[T_1]^{ij}d_iP^{k\ell}h_{k\ell j}                     \notag\\
 &+(1-a d)[T_1]^{ij}
       (\nabla_i\nabla_jP)^{k\ell}h_{k\ell}                \notag\\
 &-2a[T_1]^{ij}d_i(\nabla_jP)^{k\ell}h_{k\ell}
 -a[T_1]^{ij}(\nabla_i\nabla_jd)P^{k\ell}h_{k\ell}.
 \label{eq:weighted-newton-trace-expansion}
\end{align}
The terms containing
\[
 (\sigma_1-3\sigma_3)(1-a d)P^{k\ell}h_{k\ell}
\]
have canceled. This cancellation is the reason for applying the full
Jacobi form rather than only its second-order part.

We first control the terms containing \(\nabla h\). Since \(P\) has
constant rank two, one has
\[
                         (\nabla_iP)^{33}=0
\]
in the chosen frame. Indeed, if a local covector \(\zeta\) spans
\(\ker P\), differentiation of
\(P^{k\ell}\zeta_k\zeta_\ell=0\) gives the assertion.

For \(\alpha=1,2\), define
\begin{equation}\label{eq:projector-coefficient-definition}
 \widehat B_\alpha^{rj}=
 \begin{cases}
  2[T_1]^{ij}(\nabla_iP)^{\alpha r},
     &r=1,2,\\[2mm]
  4[T_1]^{ij}(\nabla_iP)^{\alpha3},
     &r=3,
 \end{cases}
 \qquad
 B_\alpha^{rj}
 =
 \frac12\bigl(
 \widehat B_\alpha^{rj}+\widehat B_\alpha^{jr}
 \bigr).
\end{equation}
The symmetry of \(\nabla_iP^{k\ell}\) and the full symmetry of
\(h_{ijk}\) give
\begin{equation}\label{eq:projector-coefficient-identity}
 2[T_1]^{ij}(\nabla_iP)^{k\ell}h_{k\ell j}
 =
 \sum_{\alpha=1}^2\sum_{r,j}
 B_\alpha^{rj}h_{rj\alpha}.
\end{equation}

The additional third-order term produced by the weight is written in
the same form. Define
\begin{equation}\label{eq:weight-coefficient-definition}
 \widehat D_\alpha^{rj}
 =
 -2a\mu_\alpha\delta_{r\alpha}[T_1]^{ij}d_i,
 \qquad
 D_\alpha^{rj}
 =
 \frac12\bigl(
 \widehat D_\alpha^{rj}+\widehat D_\alpha^{jr}
 \bigr).
\end{equation}
Then
\begin{equation}\label{eq:weight-coefficient-identity}
 -2a[T_1]^{ij}d_iP^{k\ell}h_{k\ell j}
 =
 \sum_{\alpha=1}^2\sum_{r,j}
 D_\alpha^{rj}h_{rj\alpha}.
\end{equation}
For each \(\alpha\), the linear coefficient and quadratic weight are,
respectively,
\[
 (1-a d)B_\alpha+D_\alpha,
 \qquad
 (1-a d)\mu_\alpha.
\]

Comparing the graph connection with the Euclidean connection \(D\),
we have
\begin{equation}\label{eq:projector-connection-split}
 (\nabla_iP)^{k\ell}
 =
 (D_iP)^{k\ell}
 +Wu^k h_{ip}P^{p\ell}
 +Wu^\ell h_{ip}P^{kp}.
\end{equation}
The tensor \(DP\) is bounded by the fixed geometry of the collar.
Consequently,
\begin{align}
 \sum_{\alpha=1}^2|B_\alpha|^2
 &\leq
 C\bigl\{|[T_1]|^2+|[T_1]h|^2\bigr\},                    \label{eq:B-coefficient-norm}\\
 \sum_{\alpha=1}^2
 \left|
 \sum_{i,j}B_\alpha^{ij}h_{ij}
 \right|^2
 &\leq
 C\bigl\{|[T_1]h|^2+|[T_1]h^2|^2\bigr\}.                 \label{eq:B-coefficient-contraction}
\end{align}
Since \(0<a\leq1\), \(|Dd|\) is bounded, and the \(\mu_\alpha\) have
fixed bounds, \eqref{eq:weight-coefficient-definition} also gives
\begin{align}
 \sum_{\alpha=1}^2|D_\alpha|^2
 &\leq C|[T_1]|^2,                                       \label{eq:D-coefficient-norm}\\
 \sum_{\alpha=1}^2
 \left|
 \sum_{i,j}D_\alpha^{ij}h_{ij}
 \right|^2
 &\leq C|[T_1]h|^2.                                      \label{eq:D-coefficient-contraction}
\end{align}

Passing to a principal frame for \(h_i{}^j\) and using
\eqref{eq:component-polynomial}, we obtain
\begin{align}
 &\sum_{\alpha=1}^2
 \left|(1-a d)B_\alpha+D_\alpha\right|^2
 \leq
 C\sum_i(\sigma_1-\kappa_i)^2(1+\kappa_i^2),              \label{eq:weighted-projector-norm}\\
 &\sum_{\alpha=1}^2
 \left|
 \sum_{i,j}
 \bigl\{(1-a d)B_\alpha^{ij}+D_\alpha^{ij}\bigr\}h_{ij}
 \right|^2
 \leq C(\sigma_1-\sigma_3)^2.                             \label{eq:weighted-projector-contraction}
\end{align}
Therefore,
\begin{align}
 &\sum_{\alpha=1}^2
 \left\{
 \left|(1-a d)B_\alpha+D_\alpha\right|^2
 +
 \frac{\left|
 \sum_{i,j}
 \bigl\{(1-a d)B_\alpha^{ij}+D_\alpha^{ij}\bigr\}h_{ij}
 \right|^2}
 {1+\min_i\kappa_i^2}
 \right\}                                                  \notag\\
 &\qquad\leq
 C\left\{
 \sum_i(\sigma_1-\kappa_i)^2(1+\kappa_i^2)
 +
 \frac{(\sigma_1-\sigma_3)^2}
      {1+\min_i\kappa_i^2}
 \right\}                                                   \notag\\
 &\qquad\leq
 C(1+\kappa_1)(\sigma_1-\sigma_3),
 \label{eq:newton-projector-coefficients}
\end{align}
where the last inequality follows from
\eqref{eq:newton-square-growth} and
\eqref{eq:least-curvature-growth}.

Apply Lemma~\ref{lem:sigma2-concavity} to \(h_{ij\alpha}\),
\(\alpha=1,2\), with weights \((1-a d)\mu_\alpha\) and coefficient
matrices \((1-a d)B_\alpha+D_\alpha\).  The tangency conditions follow
from \eqref{eq:first-geometric-variation}, while the required full
symmetry of \(h_{ijk}\) follows from the Codazzi identity.
The weights \((1-a d)\mu_\alpha\) have fixed positive upper and lower
bounds. Hence
\eqref{eq:newton-projector-coefficients} controls the quadratic
third-order terms in the first line of
\eqref{eq:weighted-newton-trace-expansion}, together with both linear
third-order terms.

It remains to estimate the terms containing no derivative of \(h\).
Differentiating \eqref{eq:projector-connection-split} once more, the
only terms in \(\nabla_i\nabla_jP^{k\ell}\) containing \(\nabla h\)
are
\[
 Wu^kP^{p\ell}h_{jpi}
 +Wu^\ell P^{kp}h_{jpi}.
\]
After contraction with \([T_1]^{ij}h_{k\ell}\), these terms vanish:
\[
 2Wu^kh_{k\ell}P^{\ell p}
 [T_1]^{ij}h_{ijp}=0
\]
by \eqref{eq:first-geometric-variation}. Thus no additional
third-order term is present.

The remaining terms in \(\nabla^2P\), together with the last two terms
in \eqref{eq:weighted-newton-trace-expansion}, involve only fixed
derivatives of \(d\) and \(P\), terms from the graph connection, and
zeroth-order factors in \(h\). Since \(0<a\leq1\), their absolute
value is bounded by
\begin{align}
 C\Bigl\{&
 \operatorname{tr}[T_1]|h|
 +|[T_1]h|\,|h|
 +|[T_1]h^2|\,|h| \notag\\
 &+
 \left|\operatorname{tr}([T_1]h)\right||h|^2
 +|h|^2
 \Bigr\}
 \leq
 C(1+\kappa_1)(\sigma_1-\sigma_3).
 \label{eq:newton-zero-order-remainder}
\end{align}
Here
\[
                         \operatorname{tr}([T_1]h)=2,
\]
while \eqref{eq:component-polynomial} controls
\([T_1]h\) and \([T_1]h^2\); the remaining factors are controlled by
\eqref{eq:curvature-growth} and
\[
                         |h|\leq C(1+\kappa_1).
\]

Finally, the term
\[
 2(1-a d)\sum_{\alpha=1}^2
 \mu_\alpha\sum_p h_{\alpha p}^2
\]
in \eqref{eq:weighted-newton-trace-expansion} is nonnegative and may
be discarded. Combining Lemma~\ref{lem:sigma2-concavity} with
\eqref{eq:newton-zero-order-remainder} proves
\eqref{eq:homogenization-absorption}.
\end{proof}

\begin{corollary}\label{cor:homogenized-trace}
In a fixed boundary collar,
\begin{align}
 &F^{ij}\bigl[(\delta_{k\ell}-n_kn_\ell)v_{k\ell}\bigr]_{ij}
 +F_{u_i}\bigl[(\delta_{k\ell}-n_kn_\ell)v_{k\ell}\bigr]_i \notag\\
 &\qquad\geq
 -C(1+\kappa_1)(\sigma_1-\sigma_3).
 \label{eq:homogenized-trace-estimate}
\end{align}
\end{corollary}

\begin{proof}
Since \(n=-Dd\), one has
\((\delta_{k\ell}-n_kn_\ell)d_\ell=0\).  Hence
\eqref{eq:v-second-derivatives} gives
\begin{align}
 (\delta_{k\ell}-n_kn_\ell)v_{k\ell}
 ={}&(1-a d)(\delta_{k\ell}-n_kn_\ell)u_{k\ell} \notag\\
 &-a u(\delta_{k\ell}-n_kn_\ell)d_{k\ell}.
 \label{eq:projected-v-expansion}
\end{align}
The full linearization of the first term on the right is bounded from
below by \eqref{eq:homogenization-absorption}.  For the second term,
write \(A=(\delta_{k\ell}-n_kn_\ell)d_{k\ell}\).  The product rule
and \eqref{eq:euler-full-operator} give
\[
 \begin{aligned}
 F^{ij}(a uA)_{ij}+F_{u_i}(a uA)_i
 ={}&aA\bigl(2+F_{u_i}u_i\bigr)\\
 &+a u\bigl(F^{ij}A_{ij}+F_{u_i}A_i\bigr)
   +2aF^{ij}u_iA_j .
 \end{aligned}
\]
Because \(a u\), \(a\), and \(aDu\) are uniformly bounded, all
coefficients in this expression are controlled independently of \(a\).
Moreover, \eqref{eq:Fui} and Lemma~\ref{lem:basic-algebra} give
\[
             \sum_i|F_{u_i}|\leq C(\sigma_1-\sigma_3),\qquad
             \sum_iF^{ii}\leq C(\sigma_1-\sigma_3).
\]
Consequently the absolute value of the full linearization of the
second term in \eqref{eq:projected-v-expansion} is at most
\[
 C(1+\sigma_1-\sigma_3)
 \leq C(1+\kappa_1)(\sigma_1-\sigma_3).
\]
Combining the two estimates proves
\eqref{eq:homogenized-trace-estimate}.
\end{proof}

\section{Proof of the second-derivative estimate}\label{sec:completion}

We use two known curvature estimates. The global-to-boundary estimate
of Guan--Ren--Wang \cite[(3.26)]{GuanRenWang} gives
\begin{equation}\label{eq:GRW-reduction}
\max_{\overline\Omega}|D^2u|
 \leq C\left(1+\max_{\partial\Omega}|D^2u|\right).
\end{equation}
Their second fundamental form uses the opposite normal convention;
choosing \(-N\) gives the principal curvatures in
\eqref{eq:intro-admissible}.  For a fixed \(x_0\in\Omega\), the
required vertical translation is uniform: the gradient estimate
controls the oscillation of \(u\), so one may choose \(c,c_0>0\),
independently of \(a\), such that
\[
 \left\langle (x,u-u(x_0)-c),-N\right\rangle
       =\frac{x\cdot Du-u+u(x_0)+c}{W}\geq c_0>0
       \qquad\text{on }\overline\Omega.
\]
This translation changes neither \(D^2u\) nor the curvature equation.

For every fixed \(\delta>0\), Qiu's local interior estimate
\cite[Theorem~1]{QiuInterior}, applied after translating the centers
and dilating the ambient graphs over balls of radius comparable to
\(\delta\), yields
\begin{equation}\label{eq:Qiu-inner}
\sup_{\{d\geq\delta\}}|D^2u|\leq C_\delta.
\end{equation}
Proposition~\ref{prop:global-gradient} makes \(|h|\) and \(|D^2u|\)
uniformly comparable.

\begin{lemma}[Collar coercivity]\label{lem:collar-coercivity}
Under the hypotheses of Theorem~\ref{thm:boundary-C2}, let
\[
                         \rho=-d+\varepsilon d^2,
 \qquad
                         G=u_i\rho_i+a u
\]
be as in Proposition~\ref{prop:double-normal}.  In a sufficiently thin
boundary collar,
\begin{equation}\label{eq:collar-coercivity}
 1+\sum_iF^{ii}+F^{ij}u_{ik}u_{jk}
 \leq
 C\left\{
 F^{ij}\rho_i\rho_j+F^{ij}G_iG_j
 \right\}.
\end{equation}
\end{lemma}

\begin{proof}
By Proposition~\ref{prop:global-gradient}, the graph metric and the
Euclidean metric are uniformly equivalent.  In the fixed boundary
collar,
\begin{equation}\label{eq:rho-gradient}
 \rho_i=(-1+2\varepsilon d)d_i,
 \qquad
 |D\rho|\geq c.
\end{equation}
Moreover, direct differentiation of \(G\) gives
\begin{equation}\label{eq:residual-gradient}
 G_i
 =
 u_{ij}\rho_j+u_j\rho_{ji}+a u_i
 =
 u_{ij}\rho_j+O(1),
\end{equation}
where the last two terms are uniformly bounded by the gradient estimate,
the fixed bounds for \(D^2\rho\), and \(0<a\leq1\).

Fix a point in the collar and choose a graph-orthonormal principal
frame \(e_1,e_2,e_3\), so that \(h_{ij}=\kappa_i\delta_{ij}\). Write
\[
 \nabla_i u=e_i(u),\qquad
 \rho_i=e_i(\rho),\qquad
 G_i=e_i(G),\qquad
 D\rho=\sum_i\rho^i e_i.
\]
Here \(\rho^i\) are the Euclidean vector components of \(D\rho\);
they are not obtained by raising the index of \(\rho_i\) with \(g\).

Since the frame is orthonormal for the graph metric,
\[
 \delta_{ij}
 =
 g(e_i,e_j)
 =
 e_i\cdot e_j+(\nabla_i u)(\nabla_j u).
\]
Hence the Euclidean metric in this frame has components
\[
                         e_i\cdot e_j
 =
 \delta_{ij}-(\nabla_i u)(\nabla_j u).
\]
It follows that
\begin{align}
 \rho_i
 &=D\rho\cdot e_i \notag\\
 &=\rho^i-(\nabla_i u)
          \sum_j(\nabla_j u)\rho^j.
 \label{eq:euclidean-graph-duality}
\end{align}
Furthermore,
\begin{equation}\label{eq:graph-gradient-size}
 \sum_i(\nabla_i u)^2
 =
 |\nabla_g u|_g^2
 =
 \frac{|Du|^2}{W^2}
 =
 1-\frac1{W^2}
 \leq1-c(\Omega)<1.
\end{equation}

We next express \(G_i\) in the principal frame.  Since
\[
 D^2u(e_i,e_j)=Wh_{ij}=W\kappa_i\delta_{ij},
\]
equation \eqref{eq:residual-gradient} gives
\begin{align}
 G_i
 &=
 D^2u(e_i,D\rho)
 +D^2\rho(e_i,Du)
 +a\nabla_i u \notag\\
 &=
 W\kappa_i\rho^i+O(1).
 \label{eq:residual-principal-components}
\end{align}

In the same principal frame,
\[
 F^{ij}
 =
 \frac1W(\sigma_1-\kappa_i)\delta_{ij}.
\]
By \eqref{eq:component-polynomial},
\begin{equation}\label{eq:derivative-eigenvalue}
 \sigma_1-\kappa_i
 =
 \frac{\sigma_1-\sigma_3}{1+\kappa_i^2}.
\end{equation}
Let \(R\geq2\), to be chosen sufficiently large below.  For the
directions satisfying \(|\kappa_i|\leq R\), we obtain
\begin{align}
 F^{ij}\rho_i\rho_j
 &=
 \frac{\sigma_1-\sigma_3}{W}
 \sum_i\frac{\rho_i^2}{1+\kappa_i^2} \notag\\
 &\geq
 \frac{c(\sigma_1-\sigma_3)}{1+R^2}
 \sum_{|\kappa_i|\leq R}\rho_i^2.
 \label{eq:rho-low-curvature-control}
\end{align}
For the remaining directions, use
\eqref{eq:residual-principal-components}.  The elementary inequality
\[
                         (r+s)^2\geq\frac12r^2-s^2
\]
gives
\begin{align*}
 F^{ij}G_iG_j
 &=
 \frac{\sigma_1-\sigma_3}{W}
 \sum_i
 \frac{\bigl(W\kappa_i\rho^i+O(1)\bigr)^2}
      {1+\kappa_i^2}\\
 &\geq
 \frac{\sigma_1-\sigma_3}{W}
 \sum_{|\kappa_i|>R}
 \left\{
 \frac{W^2\kappa_i^2(\rho^i)^2}
      {2(1+\kappa_i^2)}
 -\frac{C}{1+\kappa_i^2}
 \right\}.
\end{align*}
Since \(W\) is bounded above and below and
\[
 \frac{\kappa_i^2}{1+\kappa_i^2}\geq c
 \qquad\text{when }|\kappa_i|>R\geq2,
\]
we find
\begin{equation}\label{eq:residual-high-curvature-control}
 F^{ij}G_iG_j
 \geq
 c(\sigma_1-\sigma_3)
 \sum_{|\kappa_i|>R}(\rho^i)^2
 -\frac{C(\sigma_1-\sigma_3)}{R^2}.
\end{equation}
It remains to show that the two groups of components in
\eqref{eq:rho-low-curvature-control} and
\eqref{eq:residual-high-curvature-control} cannot both be small.
Multiplying \eqref{eq:euclidean-graph-duality} by
\(\nabla_i u\) and summing over the indices for which
\(|\kappa_i|\leq R\), we obtain
\begin{align}
 &\left(
 1-\sum_{|\kappa_i|\leq R}(\nabla_i u)^2
 \right)
 \sum_j(\nabla_j u)\rho^j                                      \notag\\
 &\qquad=
 \sum_{|\kappa_i|\leq R}(\nabla_i u)\rho_i
 +\sum_{|\kappa_i|>R}(\nabla_i u)\rho^i.
 \label{eq:rank-one-scalar-control}
\end{align}
By \eqref{eq:graph-gradient-size},
\[
 1-\sum_{|\kappa_i|\leq R}(\nabla_i u)^2
 \geq
 1-\sum_i(\nabla_i u)^2
 \geq c(\Omega).
\]
Consequently, the Cauchy--Schwarz inequality gives
\begin{align}
 \left|
 \sum_j(\nabla_j u)\rho^j
 \right|
 \leq
 C(\Omega)
 \left\{
 \left(\sum_{|\kappa_i|\leq R}\rho_i^2\right)^{1/2}
 +
 \left(\sum_{|\kappa_i|>R}(\rho^i)^2\right)^{1/2}
 \right\}.
 \label{eq:rank-one-scalar-bound}
\end{align}
For \(|\kappa_i|\leq R\), equation
\eqref{eq:euclidean-graph-duality} can be rewritten as
\[
 \rho^i
 =
 \rho_i+(\nabla_i u)
       \sum_j(\nabla_j u)\rho^j.
\]
Using \eqref{eq:rank-one-scalar-bound}, we therefore obtain
\begin{equation}\label{eq:rank-one-component-control}
 \sum_i(\rho^i)^2
 \leq
 C(\Omega)
 \left\{
 \sum_{|\kappa_i|\leq R}\rho_i^2
 +
 \sum_{|\kappa_i|>R}(\rho^i)^2
 \right\}.
\end{equation}

The equivalence of the graph and Euclidean metrics, together with
\eqref{eq:rho-gradient}, gives
\[
                         \sum_i(\rho^i)^2\geq c.
\]
It follows from \eqref{eq:rank-one-component-control} that at least one
of the inequalities
\begin{equation}\label{eq:low-high-dichotomy}
 \sum_{|\kappa_i|\leq R}\rho_i^2\geq c,
 \qquad\text{or}\qquad
 \sum_{|\kappa_i|>R}(\rho^i)^2\geq c
\end{equation}
must hold.

All constants in the preceding estimates, apart from the displayed
factors involving \(R\), are independent of \(R\). Choose \(R\geq2\)
so large that the error in
\eqref{eq:residual-high-curvature-control} is at most one half of its
positive term whenever the second alternative in
\eqref{eq:low-high-dichotomy} holds, and keep this \(R\) fixed.

If the first alternative in \eqref{eq:low-high-dichotomy} holds,
\eqref{eq:rho-low-curvature-control} gives
\[
 F^{ij}\rho_i\rho_j
 \geq c(\sigma_1-\sigma_3),
\]
where \(c>0\) now includes the fixed factor \((1+R^2)^{-1}\).  If the
second alternative holds, the choice of \(R\) and
\eqref{eq:residual-high-curvature-control} give
\[
 F^{ij}G_iG_j
 \geq c(\sigma_1-\sigma_3).
\]
Thus in either case,
\begin{equation}\label{eq:collar-coercivity-reduced}
 F^{ij}\rho_i\rho_j+F^{ij}G_iG_j
 \geq c(\sigma_1-\sigma_3).
\end{equation}

Finally, \eqref{eq:three-pairwise-sums} gives
\[
                         \sigma_1-\sigma_3\geq1,
\]
and \eqref{eq:linearized-growth} gives
\[
 1+\sum_iF^{ii}+F^{ij}u_{ik}u_{jk}
 \leq C(\sigma_1-\sigma_3).
\]
Combining this with
\eqref{eq:collar-coercivity-reduced} proves
\eqref{eq:collar-coercivity}.
\end{proof}

\begin{proof}[Proof of Theorem~\ref{thm:boundary-C2}]
By the height and gradient estimates, all constants in
Sections~\ref{sec:normal} and~\ref{sec:boundary-trace} depend only on
the data. We use \(\rho\) and \(G\) from
Lemma~\ref{lem:collar-coercivity}.

The identity \eqref{eq:linearized-residual} gives the remaining
estimate. The same principal-frame argument as in
\eqref{eq:normal-matrix-bound} gives
\[
 \left|F^{ij}u_{\ell i}\rho_{\ell j}\right|
 \leq
 C\sum_i(\sigma_1-\kappa_i)|\kappa_i|
 \leq C(\sigma_1-\sigma_3).
\]
All drift terms remaining in \eqref{eq:linearized-residual}, as well as those
in the linearization of \(d\), have uniformly bounded coefficients
and are controlled directly by \eqref{eq:Fui}.  Indeed,
\[
 |[T_2]|^2=\sum_{i<j}\kappa_i^2\kappa_j^2
 \leq C(1+\kappa_1)(\sigma_1-\sigma_3)
 \leq C(\sigma_1-\sigma_3)^2,
\]
by \eqref{eq:curvature-cofactor-growth} and the lower bound in
\eqref{eq:curvature-growth}; hence
\(|F_{u_i}|\leq C(\sigma_1-\sigma_3)\).
Thus \eqref{eq:euler-full-operator},
Lemma~\ref{lem:basic-algebra}, and the bounded derivatives of \(d\)
and \(\rho\) give
\begin{align}
 |F^{ij}G_{ij}+F_{u_i}G_i|
 +|F^{ij}d_{ij}+F_{u_i}d_i|
 &\leq C(1+\sigma_1-\sigma_3).
 \label{eq:residual-distance-linearization}
\end{align}

Set
\begin{equation}\label{eq:M-boundary}
 M=\max_{\partial\Omega}
       \bigl[(\delta_{ij}-\nu_i\nu_j)v_{ij}\bigr]_+.
\end{equation}
Lemma~\ref{lem:boundary-semiconvexity} and
\eqref{eq:GRW-reduction} give
\begin{equation}\label{eq:global-from-boundary}
 1+|D^2u|\leq C(1+M)
 \qquad\text{in }\overline\Omega.
\end{equation}
It remains to bound \(M\), and we may assume \(M\geq1\).

Keep \(\varepsilon\) fixed as in
Proposition~\ref{prop:double-normal}. Choose a preliminary collar on
which \eqref{eq:convex-defining-function},
Corollary~\ref{cor:residual-collar},
Corollary~\ref{cor:homogenized-trace}, and
Lemma~\ref{lem:collar-coercivity} all hold, and on which
\[
 \frac12\leq|-1+2\varepsilon d|\leq\frac32,
 \qquad
 \frac12\leq1-a d\leq\frac32.
\]
Fix a common constant \(C\) for these estimates. Further shrinking the
collar does not change \(C\). We now choose
\[
              c_1,\quad c_2,\quad \beta_0,\quad \omega,\quad \delta
\]
in this order; \(C_\delta\) is introduced only after \(\delta\) has
been fixed.

Since
\[
 F^{ij}\rho_i\rho_j
 =(-1+2\varepsilon d)^2F^{ij}d_i d_j,
\]
Lemma~\ref{lem:collar-coercivity} allows us first to fix
\(c_1\geq9/4\) and then \(c_2>0\), both depending only on the data, so
that, for every \(\omega>0\),
\begin{equation}\label{eq:good-squares}
 \omega F^{ij}G_i G_j+c_1\omega F^{ij}d_i d_j
 \geq c_2\omega
       \bigl(1+\sum_iF^{ii}+F^{ij}u_{ik}u_{jk}\bigr).
\end{equation}
Next set
\[
                         \beta_0=\frac{\vartheta}{8}.
\]
With \(c_1,c_2,\beta_0\) fixed, choose \(\omega\) so large that
\begin{equation}\label{eq:omega-choice}
                  c_2\omega-C(1+\beta_0)
                  \geq\frac12c_2\omega.
\end{equation}
Finally, choose the collar width \(\delta\), no larger than the
preliminary width, so that
\begin{equation}\label{eq:delta-choice}
 C\omega(1+c_1)\delta
 \leq\min\{\beta_0/2,c_2\omega/4\}.
\end{equation}
Define
\begin{equation}\label{eq:auxiliary}
 \Phi=(\delta_{ij}-n_in_j)v_{ij}
      +\omega MG^2+c_1\omega Md^2-\beta_0 Md.
\end{equation}

On \(\partial\Omega\), \(d=G=0\), so
\(\Phi=(\delta_{ij}-n_in_j)v_{ij}\). If the maximum of \(\Phi\) on the
closed collar is attained on \(\partial\Omega\), then, since
\(M\geq1\), it is attained at a point where
\[
                 (\delta_{ij}-\nu_i\nu_j)v_{ij}=M.
\]
Lemma~\ref{lem:boundary-normal-trace} gives
\[
 0\leq \Phi_\nu
 \leq-(\vartheta-\beta_0)M+C.
\]
Here \(d_\nu=-1\), so the normal derivative of
\(-\beta_0Md\) is \(\beta_0M\), while those of the two squared
corrections vanish on \(\partial\Omega\).
Since \(\beta_0=\vartheta/8\), such a boundary maximum is impossible
once \(M\) exceeds a fixed constant.

Corollary~\ref{cor:residual-collar} gives \(|G|\leq Cd\).  Hence
\[
 \omega MG^2+c_1\omega Md^2-\beta_0 Md
 \leq Md\{C\omega(1+c_1)d-\beta_0\}.
\]
By \eqref{eq:delta-choice}, the correction is nonpositive on
\(\{d=\delta\}\). After \(\delta\) has been fixed,
\eqref{eq:Qiu-inner} gives
\[
 \Phi\leq(\delta_{ij}-n_in_j)v_{ij}\leq C_\delta
 \qquad\text{on }\{d=\delta\}.
\]
The constant \(C_\delta\) is used only here. Unless \(M\) is already
bounded by the boundary estimate above or by \(C_\delta+1\), the
maximum of \(\Phi\) lies in the interior of the collar.

At that point,
\[
                         F^{ij}\Phi_{ij}+F_{u_i}\Phi_i\leq0.
\]
On the other hand, direct differentiation gives
\begin{align}
 F^{ij}\Phi_{ij}+F_{u_i}\Phi_i={}&
 F^{ij}\bigl[(\delta_{k\ell}-n_kn_\ell)v_{k\ell}\bigr]_{ij}\notag\\
 &+F_{u_i}\bigl[(\delta_{k\ell}-n_kn_\ell)v_{k\ell}\bigr]_i \notag\\
 &+2\omega MF^{ij}G_i G_j\notag\\
 &+2\omega MG(F^{ij}G_{ij}+F_{u_i}G_i)\notag\\
 &+2c_1\omega MF^{ij}d_i d_j\notag\\
 &+2c_1\omega Md(F^{ij}d_{ij}+F_{u_i}d_i)\notag\\
 &\hspace{10mm}
 -\beta_0 M(F^{ij}d_{ij}+F_{u_i}d_i).                        \label{eq:Phi-linearization}
\end{align}
Corollary~\ref{cor:homogenized-trace},
\eqref{eq:global-from-boundary},
\eqref{eq:residual-distance-linearization}, and
\(|G|\leq Cd\), together with
\eqref{eq:hessian-curvature-comparison} and
\eqref{eq:linearized-growth}, imply
\[
 1+\kappa_1\leq CM,
 \qquad
 1+\sigma_1-\sigma_3
 \leq C\left(1+\sum_iF^{ii}+F^{ij}u_{ik}u_{jk}\right),
\]
and therefore
\begin{align*}
 &F^{ij}\bigl[(\delta_{k\ell}-n_kn_\ell)v_{k\ell}\bigr]_{ij}
 +F_{u_i}\bigl[(\delta_{k\ell}-n_kn_\ell)v_{k\ell}\bigr]_i\\
 &\qquad\geq
 -CM\bigl(1+\sum_iF^{ii}+F^{ij}u_{ik}u_{jk}\bigr),
\end{align*}
and
\begin{align*}
 &2\omega MG(F^{ij}G_{ij}+F_{u_i}G_i)
 +2c_1\omega Md(F^{ij}d_{ij}+F_{u_i}d_i)
 -\beta_0 M(F^{ij}d_{ij}+F_{u_i}d_i)\\
 &\qquad\geq
 -CM\{\omega(1+c_1)d+\beta_0\}
 \bigl(1+\sum_iF^{ii}+F^{ij}u_{ik}u_{jk}\bigr).
\end{align*}
Combining these inequalities with \eqref{eq:good-squares} and
\eqref{eq:Phi-linearization}, we obtain
\begin{align*}
 F^{ij}\Phi_{ij}+F_{u_i}\Phi_i
 \geq{}&M\{c_2\omega-C(1+\beta_0)
       -C\omega(1+c_1)d\}\\
 &\qquad\times
 \bigl(1+\sum_iF^{ii}+F^{ij}u_{ik}u_{jk}\bigr).
\end{align*}
By \eqref{eq:omega-choice} and \eqref{eq:delta-choice}, for
\(0<d<\delta\) the coefficient in braces satisfies
\[
 c_2\omega-C(1+\beta_0)-C\omega(1+c_1)d
 \geq\frac14c_2\omega>0.
\]
This contradicts the interior maximum inequality.  Hence \(M\leq C\).
Lemma~\ref{lem:boundary-semiconvexity} and
\eqref{eq:GRW-reduction} now prove \eqref{eq:main-estimate}.
\end{proof}

\section{Optimality, existence, and uniqueness}\label{sec:existence}

\subsection{Optimality of the volume threshold}

To prove Remark~\ref{rem:optimal-volume}, suppose that an admissible
solution exists.  The Newton--Maclaurin inequality gives
\[
 \sigma_1(\kappa)\geq\sqrt{3\sigma_2(\kappa)}=\sqrt3.
\]
Since
\[
 \sigma_1(\kappa)=\operatorname{div}\left(\frac{Du}{W}\right),
\]
the divergence theorem and the boundary condition yield
\begin{align}
 \sqrt3\,|\Omega|
 &\leq \int_\Omega
       \operatorname{div}\left(\frac{Du}{W}\right) \notag\\
 &=\int_{\partial\Omega}\frac{u_\nu}{W}
 <|\partial\Omega|.                                      \label{eq:flux-obstruction}
\end{align}
The last inequality is strict because \(|u_\nu|/W<1\).  Hence
\(R<\sqrt3\) when \(\Omega=B_R\), and the same argument applies to
the constant-Neumann problem \eqref{eq:classical-neumann}.

For the spherical caps in \eqref{eq:ball-cap}, direct differentiation
gives
\begin{align}
 (u_{R,a})_i&=\frac{x_i}{\sqrt{3-|x|^2}},                                  \notag\\
 (u_{R,a})_{ij}
 &=\frac{\delta_{ij}}{\sqrt{3-|x|^2}}
   +\frac{x_ix_j}{(3-|x|^2)^{3/2}}.                         \label{eq:cap-derivatives}
\end{align}
Consequently,
\begin{equation}\label{eq:cap-geometry}
 W[u_{R,a}]=\frac{\sqrt3}{\sqrt{3-|x|^2}},
 \qquad
 h_{ij}[u_{R,a}]=\frac1{\sqrt3}g_{ij}[u_{R,a}],
 \qquad
 h_i{}^j[u_{R,a}]=\frac1{\sqrt3}\delta_i{}^j.
\end{equation}
Thus \(\sigma_2(h_i{}^j[u_{R,a}])=1\), and on \(\partial B_R\),
\[
 (u_{R,a})_\nu=\frac{R}{\sqrt{3-R^2}}=-a u_{R,a}.
\]
Moreover,
\[
 a\|u_{R,a}\|_{C^0(\overline B_R)}
 \geq\frac{R}{\sqrt{3-R^2}}
 \longrightarrow\infty
 \qquad\text{as }R\uparrow\sqrt3 .
\]
This proves all assertions in Remark~\ref{rem:optimal-volume}.

\subsection{The continuity argument}

\begin{proof}[Proof of Theorem~\ref{thm:existence}]
Fix \(0<a\leq1\); the continuity argument below is carried out for
this fixed value of \(a\). After translating \(\Omega\), choose an
inscribed ball
\[
 \overline B_r\subset\Omega,
 \qquad 0<r<\sqrt3.
\]
For \(0\leq t\leq1\), let
\begin{equation}\label{eq:domain-path}
                         \Omega_t=(1-t)B_r+t\Omega
\end{equation}
be the Minkowski interpolation.  Since \(B_r\subset\Omega\) and
\(\Omega\) is convex,
\begin{equation}\label{eq:domain-path-inclusion}
 B_r\subset\Omega_t\subset\Omega,
 \qquad
 |\Omega_t|\leq|\Omega|<4\pi\sqrt3.
\end{equation}
The support function of \(\Omega_t\) is \((1-t)r+t\psi\), where
\(\psi\) is the support function of \(\Omega\).  Since
\(\partial\Omega\in C^{5,\alpha}\) and \(\Omega\) is uniformly
convex, \(\psi\in C^{5,\alpha}(\mathbb S^2)\), and
\[
 (1-t)r\,g_{\mathbb S^2}
 +t\bigl(\nabla_{\mathbb S^2}^2\psi+\psi g_{\mathbb S^2}\bigr)
\]
is bounded above and below by fixed positive multiples of
\(g_{\mathbb S^2}\).  Thus \(\{\Omega_t\}_{0\leq t\leq1}\) is a
compact \(C^{5,\alpha}\) family of uniformly convex domains.  We may
therefore choose diffeomorphisms
\[
 \Phi_t:\overline B_r\longrightarrow\overline\Omega_t
\]
with \(\Phi_0\) equal to the identity, smooth dependence on \(t\), and
uniformly bounded \(C^{5,\alpha}\) norms for both \(\Phi_t\) and
\(\Phi_t^{-1}\).

At \(t=0\), the function \(u_{r,a}\) in \eqref{eq:ball-cap} is a
strictly convex solution by \eqref{eq:cap-geometry}. Fix
\(0<\beta<\alpha\). Pullback by \(\Phi_t\) defines
\(\sigma_2^{1/2}(h_i{}^j)-1\) and the boundary condition as a
continuously differentiable map from the open set of admissible
functions in \(C^{2,\beta}(\overline B_r)\) into
\[
 C^{0,\beta}(\overline B_r)
 \times C^{1,\beta}(\partial B_r).
\]
The pulled-back boundary operator is uniformly oblique, because its
derivative with respect to the gradient is the pullback of the outer
unit normal and the family \(\Phi_t\) is uniformly regular.

At an admissible solution, the derivative of the interior component is
a positive constant multiple of the interior operator in
\begin{equation}\label{eq:continuity-linearization}
 F^{ij}v_{ij}+F_{u_i}v_i=\varphi\quad\text{in }\Omega_t,
 \qquad
 v_\nu+a v=\varphi_\partial\quad\text{on }\partial\Omega_t.
\end{equation}
Multiplication by this constant does not affect invertibility.
The system \eqref{eq:continuity-linearization} is an isomorphism.
Indeed, for
\(0\leq\theta\leq1\), consider
\[
 \bigl((1-\theta)F^{ij}+\theta\delta^{ij}\bigr)v_{ij}
 +(1-\theta)F_{u_i}v_i=\varphi,
 \qquad v_\nu+a v=\varphi_\partial.
\]
For this fixed solution, the family is elliptic with constants uniform
in \(0\leq\theta\leq1\), and the boundary operator is uniformly
oblique. The homogeneous problem has only the zero solution for every
\(\theta\):
the strong maximum principle excludes a nonconstant interior
extremum, while at a positive boundary maximum the Hopf lemma gives
\(v_\nu>0\), contrary to \(v_\nu=-a v<0\); a negative boundary minimum
is treated similarly.  At \(\theta=1\) this is the invertible
Laplace--Robin problem.  The oblique Schauder estimate and the method
of continuity \cite[Theorem~6.31]{GilbargTrudinger} now prove
invertibility at \(\theta=0\).  The implicit function theorem on the
fixed domain gives nearby \(C^{2,\beta}\) solutions.  Standard oblique
Schauder bootstrapping upgrades them to the solution class in
Theorem~\ref{thm:existence}, and hence proves openness.

For closedness, \eqref{eq:domain-path-inclusion} and
\eqref{eq:C0-estimate} first give a \(C^0\) bound for the solutions
that is uniform in \(t\).  Proposition~\ref{prop:global-gradient} then
gives the uniform first-derivative bound, and
Theorem~\ref{thm:boundary-C2} gives the uniform second-derivative
bound.  Their constants are uniform in \(t\), because the volume gap
of every \(\Omega_t\) is at least that of \(\Omega\), while the family
has uniform \(C^{5,\alpha}\) geometry and uniform convexity.  Only
after these estimates have been fixed do we derive uniform
ellipticity and apply global oblique regularity.

The second-derivative estimate implies
\(|\kappa_i|\leq K_0\).  For distinct \(i,j,k\), admissibility gives
\(\kappa_j+\kappa_k>0\), and
\[
 1=\kappa_i(\kappa_j+\kappa_k)+\kappa_j\kappa_k
 \leq K_0(\kappa_j+\kappa_k)
      +\frac{(\kappa_j+\kappa_k)^2}{4}.
\]
Consequently,
\begin{equation}\label{eq:uniform-newton-ellipticity}
 \kappa_j+\kappa_k
 \geq2\bigl(\sqrt{K_0^2+1}-K_0\bigr)>0.
\end{equation}
In a principal frame these three pairwise sums are the eigenvalues of
\([T_1]\).  Since \(F^{ij}=W^{-1}[T_1]^{ij}\), the gradient estimate
and \eqref{eq:uniform-newton-ellipticity} give uniform constants
\(0<c<C\) such that
\[
 c|\xi|^2\leq F^{ij}\xi_i\xi_j\leq C|\xi|^2
 \qquad\text{for every }\xi\in\mathbb R^3.
\]
Uniform bounds for \(\Phi_t\) and \(\Phi_t^{-1}\) transfer these
ellipticity constants to the pulled-back operators.

The jets of the pulled-back solutions lie in a compact subset of the
admissible region. For fixed gradient, pullback changes the Hessian
affinely. Hence, on a fixed neighborhood of this set, the operators
\(\sigma_2^{1/2}(h_i{}^j)-1\) are concave and uniformly elliptic in
the Hessian, with uniform structural constants. The boundary operators
are uniformly oblique and have uniform H\"older bounds.
Thus the conditions F1--F5 and G2--G3 of
\cite{LiebermanTrudingerOblique} hold uniformly, and
\cite[Theorem~1.1]{LiebermanTrudingerOblique} gives a
\(C^{2,\beta}\) bound independent of \(t\).

Let \(t_m\) be solvable and \(t_m\to t_\infty\).  After passing to a
subsequence, \(u_{t_m}\circ\Phi_{t_m}\) converges in \(C^2\) to a
solution of the pulled-back problem at \(t_\infty\).  The equation and
boundary condition pass to the limit.  Moreover,
\eqref{eq:uniform-newton-ellipticity} passes to the limiting principal
curvatures, and hence
\[
 \sigma_1
 =\frac12\sum_{i=1}^3(\sigma_1-\kappa_i)
 \geq3\bigl(\sqrt{K_0^2+1}-K_0\bigr)>0,
 \qquad \sigma_2=1.
\]
Thus the limit remains in \(\Gamma_2\).  Repeated linear oblique
Schauder estimates first give \(u\in C^{5,\beta}\); since
\(C^{5,\beta}(\overline\Omega_{t_\infty})\) embeds into
\(C^{4,\alpha}(\overline\Omega_{t_\infty})\), the limiting solution
has the regularity asserted in the theorem.  Hence the set of solvable
parameters is closed.  Since it is nonempty, open, and closed in
\([0,1]\), the problem is solvable at \(t=1\).

It remains to prove uniqueness.  Let \(u\) and \(\widetilde u\) be two
admissible solutions, set \(w=u-\widetilde u\), and suppose that
\(M=\max_{\overline\Omega}w>0\).  A maximum point cannot lie on the
boundary: there the one-sided derivative gives \(w_\nu\geq0\), whereas
the boundary condition gives \(w_\nu=-a w=-aM<0\).  Thus \(M\) is attained
at an interior point.

At every interior point where \(w=M\), the two gradients agree.  For a
fixed gradient, the symmetric representative
\[
 \frac1W g^{-1/2}D^2u\,g^{-1/2}
\]
depends linearly on the Hessian, and the matrix cone \(\Gamma_2\) is
convex.  The resulting segment at the contact point is compactly
contained in the admissible set; by continuity,
\(\widetilde u+\theta w\) remains admissible in a neighborhood for
every \(0\leq\theta\leq1\).  Subtracting the two equations there gives
\[
 0=\int_0^1\left\{
 F^{ij}[\widetilde u+\theta w]w_{ij}
 +F_{u_i}[\widetilde u+\theta w]w_i
 \right\}\,d\theta .
\]
This is a locally uniformly elliptic linear equation for \(w\).
The strong maximum principle shows that \(\{w=M\}\) is open in
\(\Omega\); it is also closed and nonempty.  Since \(\Omega\) is
connected, \(w\equiv M\), which contradicts \(w_\nu=-a w\) when
\(M>0\).  Therefore \(u\leq\widetilde u\).  Interchanging the two
solutions gives \(u=\widetilde u\).
\end{proof}

\subsection{The Neumann limit}

\begin{proof}[Proof of Theorem~\ref{thm:classical-neumann}]
For \(0<a\leq1\), let \(u_a\) be the solution given by
Theorem~\ref{thm:existence}, and set
\[
 v_a=u_a-\frac1{|\Omega|}\int_\Omega u_a,\qquad
 \lambda_a=-\frac{a}{|\Omega|}\int_\Omega u_a .
\]
Then \(\int_\Omega v_a=0\) and
\[
 \sigma_2(h_i{}^j[v_a])=1\quad\text{in }\Omega,\qquad
 (v_a)_\nu=\lambda_a-a v_a\quad\text{on }\partial\Omega .
\]
Since \(u_a\leq0\) and \(u_a\not\equiv0\), the height and gradient
estimates give
\begin{align*}
 \|v_a\|_{C^0(\overline\Omega)}
 &\leq\operatorname{diam}(\Omega)\|Dv_a\|_{L^\infty(\Omega)}
 \leq C(\Omega),\\
 0<\lambda_a&\leq a\|u_a\|_{C^0(\overline\Omega)}\leq C(\Omega).
\end{align*}
Theorem~\ref{thm:boundary-C2}, applied to \(u_a\), gives
\(|D^2v_a|=|D^2u_a|\leq C\), with \(C\) independent of
\(0<a\leq1\).  Thus \eqref{eq:uniform-newton-ellipticity} and the
gradient estimate give ellipticity constants independent of \(a\).
Fix \(0<\beta<\alpha\). The global oblique estimate with the \(C^0\)
term included gives
\[
 \|v_a\|_{C^{2,\beta}(\overline\Omega)}
 \leq C\bigl(1+\|v_a\|_{C^0}+|\lambda_a|\bigr)\leq C.
\]
Here the obliqueness constant is one, and the estimate requires only
an upper bound for the zeroth-order boundary coefficient
\(a\in[0,1]\), not a positive lower bound.  Hence \(C\) is independent
of \(a\).

Choose \(a_j\downarrow0\) so that \(\lambda_{a_j}\to\lambda\) and,
after passing to a further subsequence, \(v_{a_j}\to v\) in \(C^2\).
Since \(a_j\|v_{a_j}\|_{C^0}\to0\), the equation and boundary
condition pass to
\[
 \sigma_2(h_i{}^j[v])=1\quad\text{in }\Omega,
 \qquad v_\nu=\lambda\quad\text{on }\partial\Omega.
\]
The lower bound \eqref{eq:uniform-newton-ellipticity} passes to the
limit, so \(v\) remains in \(\Gamma_2\), and
\(\int_\Omega v=0\).  Repeated linear oblique Schauder estimates give
\(v\in C^{5,\beta}(\overline\Omega)\).  Since \(0<\alpha<1\), the
embedding \(C^{5,\beta}\hookrightarrow C^{4,\alpha}\) gives the
regularity asserted in the theorem.  Moreover,
\[
 \sqrt3\,|\Omega|
 \leq\int_\Omega\sigma_1\,dx
 =\lambda\int_{\partial\Omega}\frac1W\,dS,
\]
so \(\lambda>0\).

For uniqueness, suppose that \((u,\lambda)\) and
\((\widetilde u,\widetilde\lambda)\) are admissible solutions. If
\(\lambda<\widetilde\lambda\), add a constant to \(u\) so that
\(u-\widetilde u\) has a positive maximum. This maximum cannot lie on
the boundary, since
\((u-\widetilde u)_\nu=\lambda-\widetilde\lambda<0\); the interior
comparison argument from the preceding proof also excludes an
interior maximum. Hence \(\lambda=\widetilde\lambda\).

For equal Neumann constants, interior extrema are handled by the same
comparison argument. At a boundary extremum of
\(u-\widetilde u\), the two gradients agree. The local linearization
then applies, and the Hopf lemma excludes a nonconstant boundary
maximum because \((u-\widetilde u)_\nu=0\). Thus
\(u-\widetilde u\) is constant.
Therefore \(\lambda\) is unique and, for this constant, \(u\) is
unique modulo constants. Every sequence \(a_j\downarrow0\) therefore
has a subsequence converging in \(C^2\) to the same mean-zero solution and
the same Neumann constant.  Consequently,
\(v_a\to v\) in \(C^2(\overline\Omega)\) and
\(\lambda_a\to\lambda\) as \(a\downarrow0\).
\end{proof}

\section{An endpoint Alexandrov--Fenchel inequality}\label{sec:AF}

We first obtain a weighted ABP estimate for the Neumann constant.

\begin{lemma}[Two-sided estimate for the Neumann constant]
\label{lem:neumann-constant}
Under the hypotheses of Theorem~\ref{thm:classical-neumann}, the
Neumann constant satisfies
\begin{equation}\label{eq:lambda-ABP}
 \frac{\sqrt3|\Omega|}{|\partial\Omega|}
 \leq\frac{\lambda}{\sqrt{1+\lambda^2}}
 \leq\left(\frac{|\Omega|}{4\pi\sqrt3}\right)^{1/3}.
\end{equation}
\end{lemma}

\begin{proof}
Let $u$ be the solution of \eqref{eq:classical-neumann}, and let
\(\mathcal C\) be its lower contact set,
\[
 \mathcal C=\{x\in\Omega:
 u(y)\geq u(x)+Du(x)\cdot(y-x)
 \text{ for every }y\in\overline\Omega\}.
\]
For every $p\in B_\lambda(0)$, a minimum of $u(x)-p\cdot x$ cannot
occur on the boundary, since at a boundary minimum its outer normal
derivative is nonpositive, whereas
\[
                      u_\nu-p\cdot\nu
                      =\lambda-p\cdot\nu>0.
\]
Consequently,
\[
                         B_\lambda(0)\subset Du(\mathcal C).
\]
On $\mathcal C$ one has $D^2u\geq0$, and hence all graph principal
curvatures are nonnegative.  Since $\sigma_2=1$, the
Newton--Maclaurin inequality gives
\[
 \sigma_3\leq\left(\frac{\sigma_2}{3}\right)^{3/2}
              =\frac1{3\sqrt3}.
\]
The normalized gradient map satisfies
\[
 D_i\left(\frac{u_j}{W}\right)
 =\frac{u_{ik}g^{kj}}W=h_i{}^j,
 \qquad
 \det D\left(\frac{Du}{W}\right)=\sigma_3.
\]
The radial map $p\mapsto p/\sqrt{1+|p|^2}$ sends
$B_\lambda(0)$ onto
$B_{\lambda/\sqrt{1+\lambda^2}}(0)$.  The area formula, with
multiplicity, therefore gives
\[
 \frac{4\pi}{3}\frac{\lambda^3}{(1+\lambda^2)^{3/2}}
 \leq\int_{\mathcal C}\sigma_3\,dx
 \leq\frac{|\Omega|}{3\sqrt3}.
\]
This proves the upper bound in \eqref{eq:lambda-ABP}.  For the lower
bound, use
\[
 \sigma_1=\operatorname{div}\left(\frac{Du}{W}\right),
 \qquad \sigma_1\geq\sqrt{3\sigma_2}=\sqrt3,
\]
and the divergence theorem.  Since
$W^2=1+\lambda^2+|D_Tu|^2$ on the boundary,
\[
 \sqrt3|\Omega|
 \leq\lambda\int_{\partial\Omega}\frac1W\,dS
 \leq\frac{\lambda}{\sqrt{1+\lambda^2}}|\partial\Omega|.
\]
\end{proof}

\begin{proof}[Proof of Corollary~\ref{cor:AF}]
Both sides have the same scaling. We may therefore scale \(\Omega\)
so that \eqref{eq:volume-condition} holds and then use
Theorem~\ref{thm:classical-neumann}. Eliminating the middle term in
\eqref{eq:lambda-ABP} gives
\[
 |\partial\Omega|
 \geq \sqrt3(4\pi\sqrt3)^{1/3}|\Omega|^{2/3},
\]
and hence
\begin{equation}\label{eq:isoperimetric-from-neumann}
                       |\partial\Omega|^3
                       \geq36\pi|\Omega|^2.
\end{equation}
Rescaling proves the same inequality for the original domain.
Since $\partial\Omega$ is a convex two-sphere, Gauss--Bonnet gives
\[
             \int_{\partial\Omega}\sigma_2(\Pi)\,dS=4\pi,
\]
so \eqref{eq:isoperimetric-from-neumann} is precisely
\eqref{eq:AF-intro}.

Suppose equality holds. Then equality holds in both estimates in
Lemma~\ref{lem:neumann-constant}. Equality in the flux estimate gives
\(\sigma_1\equiv\sqrt3\) in \(\Omega\) and \(D_Tu=0\) on
\(\partial\Omega\). The identity
\[
 \sigma_1^2-3\sigma_2
 =\frac12\sum_{i<j}(\kappa_i-\kappa_j)^2
\]
shows that \(\kappa_1=\kappa_2=\kappa_3=1/\sqrt3\). With the sign
convention of \eqref{eq:graph-tensors},
\(D_iN=-h_i{}^jD_jX\), and hence \(X+\sqrt3N\) is constant. The graph
is therefore contained in a round three-sphere of radius \(\sqrt3\).
Because \(D_Tu=0\) on the connected boundary, its boundary is a
horizontal section of that sphere.  Its projection is a round
two-sphere, and \(\Omega\) is a ball. Conversely, after scaling, any
ball is represented by one of the spherical caps in
\eqref{eq:ball-cap}, and hence every ball attains equality.
\end{proof}

\begin{remark}\label{rem:AF-scope}
Corollary~\ref{cor:AF} is the \(n=3\), \(k=2\) case of
\cite[Theorem~2]{QiuXiaAF}; by Gauss--Bonnet it is the sharp
isoperimetric inequality. The following estimates also involve the
Neumann solution.
\end{remark}

\subsection{Solution-dependent geometric inequalities}

Throughout this subsection, \(u\) is any solution of
\eqref{eq:classical-neumann}, and \(\lambda\) is its Neumann constant.
Thus
\(W^2=1+\lambda^2+|D_Tu|^2\) on \(\partial\Omega\).

\begin{proposition}[Degree identity and comparison ball]
\label{prop:comparison-ball}
Set
\[
                 r_\lambda=
                 \frac{\sqrt3\,\lambda}{\sqrt{1+\lambda^2}}.
\]
Then
\begin{equation}\label{eq:degree-mass}
 \int_{\{|Du|<\lambda\}}\sigma_3(h_i{}^j)\,dx
 =\frac{4\pi}{3}\frac{\lambda^3}{(1+\lambda^2)^{3/2}},
\end{equation}
and, in particular,
\begin{equation}\label{eq:curvature-mass-sandwich}
 \frac{4\pi\sqrt3\,|\Omega|^3}{|\partial\Omega|^3}
 \leq\int_{\{|Du|<\lambda\}}\sigma_3(h_i{}^j)\,dx
 \leq\frac{|\Omega|}{3\sqrt3}.
\end{equation}
Moreover,
\begin{equation}\label{eq:comparison-ball}
 |\Omega|\geq |B_{r_\lambda}|,
 \qquad
 |\partial\Omega|\geq |\partial B_{r_\lambda}|.
\end{equation}
More precisely,
\begin{equation}\label{eq:comparison-excess}
 |\partial\Omega|-4\pi r_\lambda^2
 \geq \frac3{r_\lambda}
 \left(|\Omega|-\frac{4\pi}{3}r_\lambda^3\right).
\end{equation}
All constants are sharp, and simultaneous equality holds precisely for
balls.
\end{proposition}

\begin{proof}
The map \(p\mapsto p/\sqrt{1+|p|^2}\) is an
orientation-preserving diffeomorphism from \(\mathbb R^3\) to \(B_1\).
Fix \(x_0\in\Omega\).  Convexity gives
\((x-x_0)\cdot\nu>0\) on \(\partial\Omega\).
If \(|p|<\lambda\), then
\[
                   (Du-p)\cdot\nu=\lambda-p\cdot\nu>0
                   \quad\text{on }\partial\Omega.
\]
The straight boundary homotopy from \(Du-p\) to \(x-x_0\) has positive
normal component, so \(\deg(Du,\Omega,p)=1\).  If
\(|y|<\lambda/\sqrt{1+\lambda^2}\) and
\(p=y/\sqrt{1-|y|^2}\), the zeros of \(Du/W-y\) and \(Du-p\)
coincide, and their local degrees agree because the radial map has
positive Jacobian.  Moreover,
\[
 D_i\left(\frac{u_j}{W}\right)=h_i{}^j,
 \qquad
 \det D\left(\frac{Du}{W}\right)=\sigma_3(h_i{}^j).
\]
The signed degree formula gives, for
\(0<s<\lambda/\sqrt{1+\lambda^2}\),
\[
 \int_{\{|Du|/W<s\}}\sigma_3(h_i{}^j)\,dx=\frac{4\pi}{3}s^3.
\]
Letting \(s\uparrow\lambda/\sqrt{1+\lambda^2}\) and applying dominated
convergence proves \eqref{eq:degree-mass}.

For \(\kappa\in\Gamma_2\) and \(\sigma_2(\kappa)=1\),
\begin{equation}\label{eq:sigma3-global-upper}
                         \sigma_3(\kappa)\leq\frac1{3\sqrt3}.
\end{equation}
Indeed, this is immediate if \(\sigma_3\leq0\), and otherwise all
three curvatures are positive and the Newton--Maclaurin inequality
applies.  Equations \eqref{eq:degree-mass} and
\eqref{eq:sigma3-global-upper} imply
\(|\Omega|\geq4\pi r_\lambda^3/3\).  The flux estimate in
Lemma~\ref{lem:neumann-constant} is
\(r_\lambda|\partial\Omega|\geq3|\Omega|\).  These two inequalities
give \eqref{eq:curvature-mass-sandwich},
\eqref{eq:comparison-ball}, and \eqref{eq:comparison-excess}.

The equality statement follows from that of Corollary~\ref{cor:AF}.
\end{proof}

The next estimate bounds the deficit by the curvature anisotropy and
the tangential boundary gradient.

\begin{proposition}[Solution-dependent remainder estimate]
\label{prop:quantitative-remainder}
One has
\begin{align}
 |\partial\Omega|^3-36\pi|\Omega|^2
 \geq{}&36\pi\sqrt3\,|\Omega|
 \int_\Omega
 \frac{\displaystyle\sum_{i<j}(\kappa_i-\kappa_j)^2}
      {2(\sigma_1+\sqrt3)}\,dx \notag\\
 &+36\pi\sqrt3\,|\Omega|\lambda
 \int_{\partial\Omega}
 \frac{|D_Tu|^2}
 {\sqrt{1+\lambda^2}\,W
  (\sqrt{1+\lambda^2}+W)}\,dS.
 \label{eq:quantitative-isoperimetric}
\end{align}
Equality holds if and only if \(\Omega\) is a ball.
\end{proposition}

\begin{proof}
\begin{align*}
 |\partial\Omega|^3-36\pi|\Omega|^2
 &=\frac{4\pi\sqrt3|\partial\Omega|^3}{|\Omega|}
 \left\{
 \frac{|\Omega|}{4\pi\sqrt3}
 -\left(\frac{\sqrt3|\Omega|}{|\partial\Omega|}\right)^3
 \right\}\\
 &\geq\frac{4\pi\sqrt3|\partial\Omega|^3}{|\Omega|}
 \left\{
 \left(\frac{\lambda}{\sqrt{1+\lambda^2}}\right)^3
 -\left(\frac{\sqrt3|\Omega|}{|\partial\Omega|}\right)^3
 \right\}\\
 &\geq36\pi\sqrt3|\Omega|
 \left\{
 \frac{\lambda|\partial\Omega|}{\sqrt{1+\lambda^2}}
 -\sqrt3|\Omega|
 \right\}.
\end{align*}
The last factor has the exact decomposition
\begin{align*}
 \frac{\lambda|\partial\Omega|}{\sqrt{1+\lambda^2}}
 -\sqrt3|\Omega|
 ={}&\int_\Omega(\sigma_1-\sqrt3)\,dx
     +\int_{\partial\Omega}
       \left(\frac{\lambda}{\sqrt{1+\lambda^2}}
             -\frac{\lambda}{W}\right)dS\\
 ={}&\int_\Omega
 \frac{\displaystyle\sum_{i<j}(\kappa_i-\kappa_j)^2}
      {2(\sigma_1+\sqrt3)}\,dx\\
 &+\lambda\int_{\partial\Omega}
 \frac{|D_Tu|^2}
 {\sqrt{1+\lambda^2}\,W
  (\sqrt{1+\lambda^2}+W)}\,dS.
\end{align*}
This proves \eqref{eq:quantitative-isoperimetric}.  In the equality
case both bounds in Lemma~\ref{lem:neumann-constant} are equalities;
the rigidity argument in the proof of Corollary~\ref{cor:AF} then
shows that \(\Omega\) is a ball.  Balls give equality.
\end{proof}

The flux identity also gives a sharp weighted boundary product.

\begin{proposition}[Weighted boundary inequality]
\label{prop:weighted-HK}
Let \(H_{\partial\Omega}=\sigma_1(\Pi)\).  Then
\begin{equation}\label{eq:weighted-HK}
 \left(\int_{\partial\Omega}
       \frac{\lambda^2H_{\partial\Omega}}{W^2}\,dS\right)
 \left(\int_{\partial\Omega}
       \frac1{H_{\partial\Omega}}\,dS\right)
 \geq3|\Omega|^2.
\end{equation}
The constant is optimal, and balls attain equality.
\end{proposition}

\begin{proof}
The Cauchy--Schwarz inequality and the flux identity give
\begin{align*}
 \left(\int_{\partial\Omega}
       \frac{\lambda^2H_{\partial\Omega}}{W^2}\,dS\right)
 \left(\int_{\partial\Omega}
       \frac1{H_{\partial\Omega}}\,dS\right)
 &\geq\left(\lambda\int_{\partial\Omega}\frac1W\,dS\right)^2\\
 &=\left(\int_\Omega\sigma_1\,dx\right)^2
 \geq3|\Omega|^2.
\end{align*}
The spherical caps over balls realize equality.
\end{proof}

\section*{Acknowledgments}

The author acknowledges support from Grant 2025YFA1017603 of the
National Key R\&D Program of China and Grant 12571227 of the National
Natural Science Foundation of China.

\end{document}